\documentclass[11pt,leqno]{article}
\usepackage{anysize}
\marginsize{3.5cm}{3.5cm}{3cm}{3cm}

\usepackage[]{hyperref}
\hypersetup{
    colorlinks=true,       
    linkcolor=red,          
    citecolor=blue,        
    filecolor=magenta,      
    urlcolor=cyan           
}

\usepackage{amsmath}
\usepackage{amsfonts,amssymb}
\usepackage{enumerate}
\usepackage{mathrsfs}
\usepackage{amsthm}
\usepackage{xcolor}
\usepackage{tikz}
\usepackage{textcomp}

\usepackage{ucs}
\usepackage{textcomp}
\usepackage{dsfont}

\theoremstyle{plain}
\newtheorem{theo}{Theorem}[section]
\newtheorem*{theo*}{Theorem}
\newtheorem{prop}[theo]{Proposition}
\newtheorem{lemm}[theo]{Lemma}

\theoremstyle{definition}
\newtheorem{rema}[theo]{Remark}

\DeclareMathOperator{\diff}{d}
\DeclareSymbolFont{pletters}{OT1}{cmr}{m}{sl}
\DeclareMathSymbol{s}{\mathalpha}{pletters}{`s}
\DeclareMathOperator*{\argmin}{arg\,min}

\def \Cof {\operatorname{Cof}}

\def\ba{\begin{align}}
\def\bad{\begin{aligned}}
\def\be{\begin{equation}}
\def\ea{\end{align}}
\def\ead{\end{aligned}}
\def\ee{\end{equation}}

\def \D{\diff\! }

\def\dx{\diff \! x}

\def\le{\leq}

\def \I{\textrm{I}}
\def \Tau{\mathbf{T}}
\def \Id{\textrm{Id}}
\def \bt{\mathbf{\boldsymbol{\tau}}}
\def \bH{\mathbf{H}}
\def \bn{\mathbf{n}}
\def \bx{\mathbf{x}}
\def \by{\mathbf{y}}
\def \bu{\mathbf{u}}
\def \bU{\mathbf{U}}
\def \bv{\mathbf{v}}
\def \bpsi{\mathbf{\psi}}
\def \R{\mathbb{R}} 
\def \bI{\mathbf{I}}

\def \N{\mathbf{N}}

\def \Hu {\mathbf{H}_\varphi}
\def \Hup {\widetilde{\Hu}}

\def\BV{\mathrm{BV}}
\def\Div{\mathrm {div}\,}

\numberwithin{equation}{section}

\title{Implicit like time discretization for the one-phase Hele-Shaw problem with surface tension}
\author{Ido Lavi \thanks{Departament de F\'isica de la Mat\`eria Condensada, Universitat de Barcelona, Avinguda Diagonal 647, 08028 Barcelona, Spain},
Nicolas Meunier \thanks{LaMME, UMR CNRS 8071, Universit\'e \'Evry Val d’Essonne, France,  nicolas.meunier@univ-evry.fr},
Olivier Pantz \thanks{La­bo­ra­toire Jean Alexandre Dieu­don­n\'e, UMR CNRS 7351,  Universit\'e de Nice, olivier.pantz@unice.fr }}

\begin{document}

\maketitle

\begin{abstract}
In this work, we propose and compare three numerical methods to handle the one-phase Hele-Shaw problem with surface tension in dimension two by using three variational approaches in the spirit of the seminal works \cite{Otto, Gia_Otto}. 
\end{abstract}

%


\section{Introduction}
Consider the classical experiment: a droplet of viscous fluid  is trapped between two narrowly spaced horizontal glass plates (the Hele-Shaw cell see Fig. \ref{fig:H-S}). Because of this special geometry, the motion of the viscous fluid is strongly overdamped; it is assumed to be governed by Darcy's law. The effect of surface tension at the interface is such that the droplet is at rest if and only if its cross-section is circular.
\begin{figure}
	\begin{center}
		\includegraphics[scale=0.16]{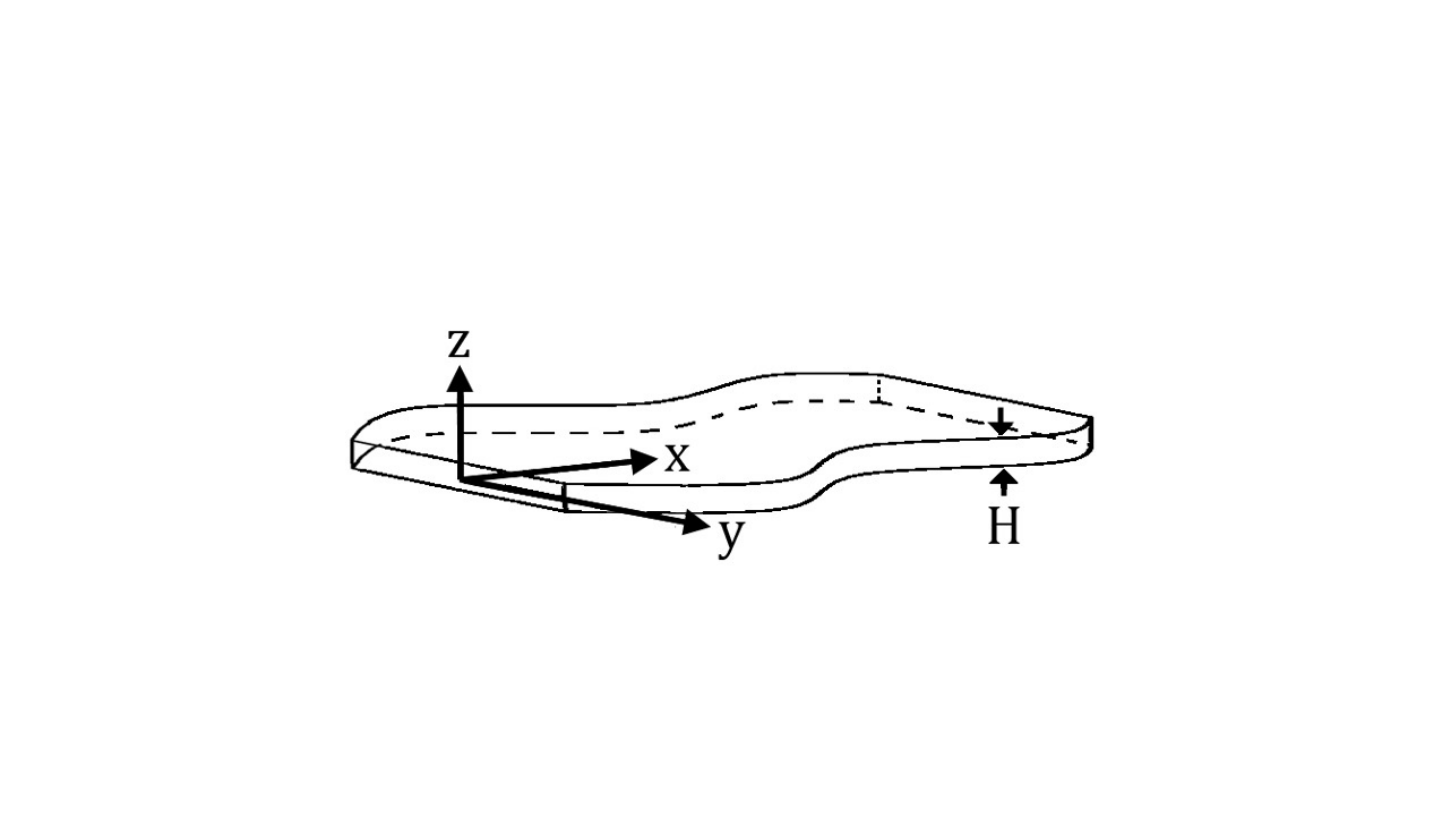}
		\caption{Viscous droplet confined between two parallel plates.}
		\label{fig:H-S}
	\end{center}
\end{figure}

In equations, the continuous time moving-boundary-value problem associated with the confined droplet is the following. Let $\Omega_0=\Omega(t=0)$ be a given smooth (convex and bounded) domain of $\mathbb{R}^2$. For all $t > 0$, the velocity $\bu(t, \bx)$, the pressure $p(t, \bx)$ and the domain describing the droplet $\Omega(t)$ together with its boundary $\Gamma(t):=\partial \Omega(t)$ satisfy
\[
\bu + \nabla p=0,
\]
with $p$ solution of 
\begin{equation}\label{eq:Laplace_mbv}
	\begin{cases} 
		\Delta p=0 & \text{in }\Omega(t) \\
		p=\sigma\kappa & \text{on }\Gamma(t)
	\end{cases}
\end{equation}
together with the kinematic condition, stating that the sharp interface is transported by the fluid 
\begin{equation}\label{eq:kin}
V_\text{n}=V_{\Gamma(t)}\cdot \bn=\bu\cdot\bn \quad \text{on }\Gamma(t),
\end{equation}
where $\kappa$ is the (mean) curvature (positive for a circle) of the evolving free-boundary $\Gamma(t)$, $V_n$ and $\bn$ are the normal velocity field and the outer unit normal field over $\Gamma(t)$, and $\sigma > 0$ the surface tension is a given constant.  

Equations \eqref{eq:Laplace_mbv} -- \eqref{eq:kin} describe a situation where a partial differential equation is solved for an unknown function, the fluid velocity $\bu(\bx,t)$ and the pressure $p(\bx,t)$, but at the same time the exact extent of the computational domain is also unknown. This is a free boundary problem. Consequently, one of the computational tasks will be to treat the domain shape throughout the simulation.

It is well-known that the perimeter $\mathcal P(\Omega)$ is a Liapunov functional for the evolution problem \eqref{eq:Laplace_mbv} -- \eqref{eq:kin}, see \cite{AMS} e.g., that is 
\[
\frac{\D }{\D t} \mathcal P(\Omega(t))= -\frac1\sigma \int_{\Omega(t)} |\nabla p|^2 \dx.  
\]
Furthermore, in \cite{Otto, Gia_Otto}, it was proved that the evolution problem \eqref{eq:Laplace_mbv} -- \eqref{eq:kin} can be understood as the gradient flow of the perimeter. It can be approximated by  the so-called JKO scheme, which, starting from an initial data $\Omega_0$, is given by a sequence $k \in \N$ of variational problems of the form
\begin{align}
	& \Omega^\tau_{(k+1)} \textrm{ minimizes }\nonumber \\
	& 
	\qquad \qquad \qquad W_2^2(\chi_\Omega,\chi_{\Omega^\tau_{(k)}})   + 2\tau \mathcal P(\Omega) 
	\nonumber \\
	& \qquad \qquad \qquad \mbox{ among all the sets of finite perimeter with }  |\Omega|=|\Omega_0|, 
	\label{eq:JKO}
\end{align}
where $k$ is the time step, $\tau$ is its size,  $\chi_{\Omega} \in \BV (\Omega ; \{0,1\})  $ is the characteristic function of the set $\Omega$, that is
\[
\chi_{\Omega}(x)=
\begin{cases}
	1 & \textrm{ in }  \Omega\, ,\\
	0 & \textrm{ else}\, ,
\end{cases}
\]
$|\Omega|$ is the area of $\Omega$ and $W_2$ is the usual Wasserstein distance, defined by
\begin{eqnarray}\label{def:wasserstein}
	W^2_2\left( \chi_{\Omega}-\chi_{\Omega'} \right) &=&  \inf_{\pi \in \Pi (\mu , \nu ) } \iint _{\Omega \times \Omega} |x - y|^2\D\pi (x, y)\nonumber \\
	&=&\inf \int_{\Omega} |x- T(x) |^2 \D x 
\end{eqnarray}
where $\Pi (\mu , \nu )$ denotes the set of all probability measures $\pi \in \mathbb P(\Omega \times \Omega)$ with marginals $\mu$ and $\nu$ and where in the second line the infimum runs over all transport maps, i.e., measure preserving  $T: \R ^2 \to  \R^2$ such that $T _{\sharp \chi _{\Omega}} = \chi_{\Omega '}$, that is
\[
\int_{\R^2} T(x)  \chi_{\Omega}(x)\D x=\int_{\R^2} \chi_{\Omega'}(y)\D y\, .
\]

Under the assumption that the perimeter $\mathcal P(\Omega_k)$ converges toward the perimeter $\mathcal P(\Omega)$, by using the JKO scheme defined by (1.14) in \cite{Gia_Otto}, it was possible to prove the existence of weak solution to \eqref{eq:Laplace_mbv} -- \eqref{eq:kin}, see \cite{Otto, Gia_Otto}. Therefore, the natural numerical method would be to use this scheme. But it is very nonlinear. In this work, we present numerical discretziation schemes by using sequences of variational problems in the spirit of \eqref{eq:JKO} but which are less nonlinear and we compare their results.
From a practical point of view, there are several difficulties to solve:
\begin{enumerate}
\item Evaluation of the Wasserstein distance.
\item Which discretization in space?
\item How to impose the incompressibility constraint?
\end{enumerate}

Throughout this work, we will consider smooth domains and we will use for the Wasserstein distance the formulation that is based on the minimization of a diffeomorphism between the two forms whose distance we want to evaluate.

Let us give a brief overview of the content of this article. In Section \ref{sec:diff_scheme} we present the different time discrete problems associated with \eqref{eq:mbv} -- \eqref{eq:kin_mod}.  In Section \ref{sec:def}, we recall some basic facts that will be useful. In Sections  \ref{sec:explicit} and \ref{sec:var_2}, we study the explicit scheme and \eqref{pb:well-posed} -- \eqref{pb:well-posed_2}. In Section \ref{sec:stab_scheme} we study \eqref{eq:Otto_3} -- \eqref{eq:func_J_alpha} and in Section \ref{sec:nonlinear} we study \eqref{eq:Otto_4}. Finally we end with numerical simulations in Section \ref{sec:verification}.

\section{Different time-discretisations of \eqref{eq:mbv} -- \eqref{eq:kin_mod}} \label{sec:diff_scheme}

In free boundary problems, representations of interface and bulk fields are coupled through: 
(i) interface kinematics: the transport of the Lagrangian or Eulerian interface description by the Eulerian velocity field, 
or (ii) interface dynamics: problems with boundary (or jump) conditions associated with the sharp interface. 
The solution of the interface kinematic problem has seen major progress in the past twenty years (see \cite{Peskin, Hirt}). However, the situation for interface dynamics, and surface tension in particular, is more complex and a wide range of methods or their combinations have been proposed, see, e.g., \cite{Gallinato}.

Typically, Hele-Shaw flow problems with a sharp moving interface are solved numerically (if not analytically) by taking advantage of elegant meshless techniques such as conformal mapping and the vortex-sheet method (see \cite{Dallaston} for a review). 

However, to simulate the type of problem that is involved in some recent developments in mathematical biology, in which the fluid flow ($\bu$) is coupled to the bulk dynamics of some concentration field ($c$), see e.g. \cite{Aranson, AlaMaMeu, Berlyand2, Berlyand_2021, Blanch,  CMM2, LMVC}, 
it is crucial to determine, at each time step, the deformed geometry on which $c$ and $\bu$ are defined. In principle, one could avoid explicitly tracing the interface by using a phase-field (or level-set) method, but this entails formulating a new model which would only approximate our equations of motion at a computationally-expensive limit \cite{CMM1}. 
In our attempt to maximize precision at efficient computation costs, we chose to build a dynamic-interface simulation based on the finite element method (FEM), which is known as one of the most powerful numerical techniques for solving PDEs on arbitrary domains. In a further work \cite{LMP2}, we will apply the three methods described here to the problems described in \cite{Aranson, AlaMaMeu, Berlyand2, Berlyand_2021, Blanch,  CMM2, LMVC}.

Rewriting the equation  \eqref{eq:Laplace_mbv} in a slightly different form
\begin{equation}\label{eq:mbv}
	\begin{cases} 
		\bu+\nabla p=0 & \text{in }\Omega(t) \, ,\\
		\nabla\cdot\bu=0 & \text{in }\Omega(t) \, ,\\
		p=\sigma\kappa & \text{on }\Gamma(t)\, ,
	\end{cases}
\end{equation}
allows us to derive a mixed variational formulation of the PDE model in which the pressure boundary condition is included as a natural boundary condition.

Let us first detail the role played by the normal component of the velocity on the boundary. 
Let $ \mathbf{X} ( \cdot , t) : I \to \Gamma(t)$ be a parameterization of $\Gamma(t)$.
The kinematic condition \eqref{eq:kin} writes
\begin{equation}\label{eq:kin_normal}
	\begin{cases}
		\mathbf{X}(\cdot, t) &= \Gamma(t)\, ,\\
		\frac{\D \mathbf{X}(s, t)}{\D t}&= \left(\bu \left( \mathbf{X}(s, t) \right) \cdot \bn \left( \mathbf{X}(s, t) \right)   \right) \bn \left( X(s, t) \right)  + f( s, t ) \bt \left( X(s, t) \right)\, ,
	\end{cases}
\end{equation}  
where $s$ is the parameterization variable, $I$ is the parameterization interval, $\bn$ is the outwards pointing normal vector of $\Gamma$, $\bt$ is the tangent vector, and $f( s, t )$ is any smooth function. According to \eqref{eq:kin_normal}, the time-varying set of points $\Gamma(t)$ only depends on the normal component of the velocity ($\bu\cdot\bn$). Indeed, any movement along the tangent $\bt$ simply serves to re-parameterize $\Gamma$. In other words, the shape of the deformed interface is determined by the normal velocity alone. Therefore, the tangential flow component on the boundary ($\bu\cdot \bt$) is completely irrelevant for the continuous problem. In fact, we can modify the equation of motion of $\Gamma(t)$ so that it is
\begin{equation}\label{eq:kin_mod}
	\frac{\D \mathbf{X}(s, t)}{\D t}= \bu \left( X(s, t)\right) .
\end{equation}  
Much of the computational accuracy and stability depends on the discretization of the term responsible for the surface tension force in \eqref{eq:mbv}. For most numerical approaches  the evaluation of the curvature $\kappa \bn$ is very difficult as it contains second derivatives.

\subsection{Explicit time-discretisation of \eqref{eq:mbv} -- \eqref{eq:kin_mod}}

Knowing the domain configuration at $t = t_n$, and seeking the solution at $t_{n+1}$ of \eqref{eq:mbv}, an explicit treatment of the surface tension term implies that the curvature vector $ \bH:=\kappa \bn$ is obtained using the known domain configuration $X(\cdot,t_{n})$.  For simplicity, we only consider one time-step $\Delta t>0$ of the time-discrete problem. This allows us to drop the time index notation and we omit the superscript $i$ for brevity.

More precisely, let $\Omega$ and $\Gamma$ be the domain and its interface at time $t_n=n\Delta t$, $n \in \N$. The explicit time-discrete version of \eqref{eq:kin} -- \eqref{eq:mbv} is: find $\bu\in \left(L^2(\Omega)\right)^2$ and $p\in H^1 (\Omega)$ solutions of
\begin{equation}\label{eq:mbv discrete_explicit}
	\begin{cases} 
		\bu+\nabla p=0 & \text{in }\Omega\, , \\
		\nabla\cdot\bu=0 & \text{in }\Omega\, ,\\
		p\bn=\sigma
		\bH& \text{on }\Gamma\, ,
	\end{cases}
\end{equation}
where $\bH$ and $\bn$ are the curvature vector and the unit outwards normal vector of $\Gamma$.

According to the kinematic condition, \eqref{eq:kin}, the new domain is defined by
\begin{equation}\label{eq:discrete propagation}
	\Omega_{\varphi}=(\Id+\Delta t \, \bu)(\Omega):=\varphi(\Omega) \, ,
\end{equation}	
and 
\begin{equation}\label{eq:discrete propagation_boundary}
	\mathbf{X}_{\varphi}(s):= \varphi  \left(\mathbf{X}(s)\right) \, ,
\end{equation}
where $\bu$ has been extended to $\overline{\Omega }$. 
Note here that $s=\left(\mathbf{X}\right)^{-1}( \mathbf{x})$ for all $\mathbf{x} \in \Gamma $, hence \eqref{eq:discrete propagation_boundary} re-writes as
\begin{equation}\label{eq:discrete propagation_boundary_2}
	\mathbf{X}_{\varphi}\left( \left(\mathbf{X}\right)^{-1}( \mathbf{x})\right) = \mathbf{x} +\Delta t \, \bu(\mathbf{x}) \, , \quad \textrm{ for all } \mathbf{x} \in \Gamma \, .
\end{equation}

The use of a finite time-step for the new position of the interface introduces a time-discretization error (that is a difference between $\mathbf{X}_{\varphi}$ given by \eqref{eq:discrete propagation_boundary} and the exact position where the interface "should" be for the continuous time case \eqref{eq:kin}) which is directly related to the values of $\Delta t$ and of the velocity $\bu$.

In \eqref{eq:mbv discrete_explicit}, the unknowns ($\bu$, $p$) are explicit, meaning that their computation is done from the geometry of the previous time, that is $\Omega$. 

The solution to \eqref{eq:mbv discrete_explicit} is the solution of the minimization problem 
\[
\inf _{\bu \in V(\Omega)} J_{\rm explicit}(\bu), 
\]
where 
\begin{equation}
	J_{\rm explicit}(\bu)=\frac{1 }{2}\int_\Omega |\bu|^2 \D x +\frac{\sigma}{\Delta t} \mathcal P (\Omega)+ \sigma \langle\mathcal P' (\Omega), \bu \rangle \, ,
	\label{eq:functional}
\end{equation}
and $V(\Omega)$ is the functional space
\begin{equation}\label{def:V}
	V(\Omega) = \left\{  \bu\in \left(L^{2} \left(\Omega \right) \right) ^2 \textrm{ such that } \Div\,  \bu=0
	\right\}\, .
\end{equation}

\subsection{Other formulations}

The main result of the present work is the construction of numerical solutions in the spirit of the implicit time discretization \eqref{eq:JKO} proposed in \cite{Otto,Gia_Otto}. To do so, we consider three different directions: a fully boundary minimization problem, a curl penalization and a fully non-linear treatment of the incompressibility constraint.

\subsubsection{Boundary variation problem}

The boundary minimization problem is:
\begin{equation}\label{pb:well-posed}
	\inf_{\psi\in BV(\Gamma)} J_{\rm mod}(\psi),
\end{equation}
with
\begin{equation}\label{pb:well-posed_2}
J_{\rm mod}(\psi) :=\|\psi\cdot \bn \|_{H^{-1/2}(\Gamma)} + V\left((\Id +\Delta t \, \psi)(x)\right),
\end{equation}
where $V(\cdot)$ is the total variation and $BV$ is the space of functions with bounded variations.

\subsubsection{A problem with curl penalization}
 
The minimization problem with a curl penalization is:
\begin{equation}\label{eq:Otto_3}
	\inf _{\bu \in V(\Omega) } J_\alpha (\Omega,\bu)\, ,
\end{equation}
with $\alpha >0$ and
\begin{equation}\label{eq:func_J_alpha}
	J_\alpha (\Omega,\bu)= \frac 1 2 \int_\Omega |\bu|^2 \, \dx + \alpha \int_\Omega |\nabla \wedge \bu|^2 \, \dx  + \frac \sigma{\Delta t} \mathcal P\left((\Id+\Delta t \, \bu)(\Omega) \right)\, .
\end{equation}

\begin{rema}
	As we will see below, one can not take $\alpha=0$. Indeed the minimization problem 
	\begin{equation}\label{eq:Otto_2}
		\inf _{\bu \in V(\Omega )} J_0(\Omega,\bu)\, ,
	\end{equation}
	is ill-posed. 	
\end{rema}	
\subsubsection{A fully  nonlinear problem}

The fully nonlinear minimization problem is
\begin{equation}\label{eq:Otto_4}
	 \inf _{\bu \in W(\Omega) } J(\Omega,\bu)\, ,
\end{equation}
where $W(\Omega) $ is the functional space
\begin{equation}\label{espace:incompressible}
W(\Omega) = \left\{ \bu\in \left(H^{1} \left(\Omega \right) \right) ^2 \textrm{ such that } \det\left( \Id + \Delta t \,  \nabla \bu \right)=1 \text{ in }\Omega
\right\}\, .
\end{equation}

\section{Some material}\label{sec:def}

In this part, we recall some very classical facts on differential geometry, on eulerian derivatives and on functions with bounded variations that will be used later on.

\subsection{Computation of the first variation of the deformed permiter}

In order to define the eulerian derivative of the perimeter we embed problem \eqref{eq:mbv}  into a family of perturbed problems which are defined on perturbations of a $\mathcal C^{2,1}$ reference domain $\Omega$ constructed by perturbing the identity.

 Let $U$ be a convex bounded domain of class $\mathcal C^{2,1}$ such that $\bar \Omega \subset U$ and let
 \begin{equation}\label{def:esp_def}
 \mathcal S=\{\bu \in \mathcal C^{2,1}(\bar U ,\R^2)\, :\, \bu =0 \textrm{ on } \partial U  \textrm{ and } \bu \cdot \bn =0 \textrm{ on } \Gamma \} 
 \end{equation}
 be the space of feasible deformation fields endowed with the natural norm in $C^{2}(\bar U ,\R^2)$. 
 For a fixed field $\bu \in \mathcal S$ and for all $\Delta t>0$ define the mapping from $\bar U$ to $\R^2$ by  $\varphi=\Id+\Delta t \bu$. For $\Delta t$ sufficiently small $\varphi$ defines a family of $\mathcal C^2$-diffeomorphisms of $U$ onto itself. For such $\Delta t$ one sets
\begin{equation}\label{eq:discrete propagation_10}
	\Omega_\varphi=\varphi(\Omega)=(\Id+\Delta t \, \bu)(\Omega) \subset \R^2\, ,  
\end{equation}
and
\[
\Gamma_{\varphi} = \varphi(\Gamma),
\] 
hence $\Omega_0 = \Omega$, $\Gamma_0 = \Gamma$ and $\partial \varphi (\Omega):=\varphi(\partial \Omega)$.

We define the matrix  $\nabla \varphi$ by $\left(\nabla \varphi \right)_{ij}=\partial _j \varphi_i$ for $1\le i,j\le 2$.
This matrix represents the differential of $\varphi$ in the sense that $\varphi (\bx + h) - \varphi(\bx) = \nabla \varphi(\bx)h + o(\|h\|)$ for $\bx$ and $\bx + h$ in $\Omega$.

\subsubsection{On the curvature vectors of $\partial \Omega$ and of $\varphi (\partial \Omega)$}

Let $\bt$ and $\bt_\varphi $ denote the tangent vectors along $\partial \Omega$ and $\partial \varphi (\Omega)$ respectively, and let $\Tau$ represent the tangent vector on $\partial \varphi (\Omega)$ "pulled back" to the reference domain $\partial \Omega$. In other words, the input of $\Tau$ is any $\mathbf{x}$ on $\partial \Omega$ and the output is $\bt_{\varphi}(\mathbf{x}_\varphi)$ where $\mathbf{x}_\varphi:=\bx+\Delta t \, \bu(\bx)$ and 
\begin{equation}
	\Tau:=\bt_\varphi \circ \varphi =\frac{\nabla \varphi\, \bt}{|\nabla \varphi \,\bt|} =\frac{(\I+\Delta t \nabla\bu)\bt}{|(\I+\Delta t \nabla\bu)\bt|} \, ,\label{eq:Tau i+1}
\end{equation}
where  $\I$ represents the $2\times2$ identity matrix and we use the notation $\nabla \varphi\, \bt = (\bt\cdot\nabla)\varphi$.

Let $\mathbf{X}(s)$ for $s\in I$ be a parameterization of $\partial \Omega$ and $\mathbf{X}_\varphi(s)$ be a parametrization of the interface $\partial \varphi (\Omega)$:
\begin{equation}
	\mathbf{X}_\varphi(s)=\varphi (\mathbf{X}(s))=\mathbf{X}(s)+\Delta t \, \bu\left(\mathbf{X} (s)\right) \, .\label{eq:s parameterization}
\end{equation}

Using differential geometry tools \cite{Walker}, the curvature vectors $\bH:= \kappa\bn$ on $\partial \Omega$ and $\bH_\varphi:= \kappa_\varphi\bn_\varphi$ on $\partial \varphi (\Omega)$ are defined by
\[
\bH=-|\mathbf{X}^{'}(s)|^{-1}\frac{\D }{\D s}\bt(\mathbf{X}(s)) =- \nabla \bt \, \bt\, ,
\]
\[
\bH_\varphi=-|\mathbf{X}_\varphi^{'}(s)|^{-1}\frac{\D }{\D s}\bt_\varphi(\mathbf{X}_\varphi(s)) = -\nabla \bt_\varphi\, \bt_\varphi\, .
\]

Define the vector $\Hup$ by
\begin{equation}
	\begin{aligned}
		\Hup
		&=-|\mathbf{X}^{'}(s)|^{-1} \frac{\D }{\D s} \left(\Tau(\mathbf{X}(s))
		\right)=-\nabla \Tau\, \bt \, .
	\end{aligned} \label{eq:Implicit curvature vector}
\end{equation}

Let $\Cof \nabla\varphi$ be the cofactor matrix associated with the gradient matrix $\nabla\varphi$. Since
$$
\Cof \nabla\varphi^T \nabla\varphi=\det(\nabla \varphi) \Id \, ,
$$
we see that 
$$
\Hup=|\Cof \nabla \varphi| \, \bH_{\varphi}\circ \varphi\, .
$$

\subsubsection{First variation of the perimeter $\mathcal P \left(\varphi (\Omega)\right)$}

The Eulerian derivative of the permiter $\mathcal P$ at $\Omega$ in the direction $\bu$ is defined as
\begin{equation}\label{def:shape_der_per}
	\langle  \mathcal P' \left(\Omega\right); \bu \rangle  :=\lim_{\Delta t \to 0}\frac{\mathcal P \left(\varphi (\Omega)\right)-P \left(\Omega\right) }{\Delta t}\, .
\end{equation}

In order to evaluate the first variation of $P(\Omega_\varphi)$ we transform $P(\Omega_\varphi)$ to an integral over the reference domain $\Omega$. The general strategy we use in this work is adapted from the one used for shape derivative calculations. It consists in transferring the problem on the original boundary before writing the variational problem.

\begin{lemm}\label{lem:shape_der_per}
	The first variation of the perimeter 
	$\mathcal P \left(\varphi (\Omega)\right)$ is given by
	\begin{equation}\label{eq:der_per}
		\langle \mathcal P' \left(\varphi (\Omega)\right); \bv \rangle = \Delta t \int_{\partial \varphi(\Omega)} \bH_{\varphi}\cdot (\bv\circ \varphi^{-1}) \D s _\varphi= \Delta t \int_{\partial\Omega} \Hup \cdot \bv \D s \, .
	\end{equation}
\end{lemm}
\begin{proof}  
	Assume that the boundary $\partial \Omega$ is parametrized by the arc length $s$, we have 
	\begin{equation}\label{def:perm_def}
	\mathcal P \left(\varphi (\Omega)\right)=\int_{\partial \Omega} | \nabla \varphi \,\bt| \D s= \int_{\partial \Omega} |(\I+\Delta t \nabla \bu)\bt| \D s\, .
	\end{equation}

	Let us perform an asymptotic expansion of order one of $\mathcal P \left(\varphi (\Omega)\right)$. 
	We first find 
	\begin{eqnarray*}
		& &|(\I+\Delta t\, \nabla(\bu+\delta\bu))\bt|^2 \\
		&=&  |(\I+\Delta t\, \nabla\bu)\bt|^2+2\Delta t \nabla \delta \bu\, \bt \cdot (\bt+\Delta t \, \nabla\bu\, \bt) +O(||\nabla\delta\bu||^2)\, .
	\end{eqnarray*}
	Let $\D S= |(\I+\Delta t\, \nabla\bu)\bt|$, using the Taylor expansion of $\sqrt{\D S^2+x}$, we obtain
	\[
	|(\I+\Delta t\, \nabla(\bu+\delta\bu))\bt|=|(\I+\Delta t \nabla \bu)\bt|+\Delta t \nabla \delta \bu \,\bt\cdot \Tau +O(||\nabla \delta \bu ||^2)\, ,
	\]
	where we recall the definition of $\Tau$:
	\[
	\Tau=\frac{\nabla \varphi\, \bt}{|\nabla \varphi\, \bt|}=\frac{(\I+\Delta t \nabla \bu)\bt}{|(\I+\Delta t \nabla \bu)\bt|}\, ,
	\]
	and thus 
	\[
	\mathcal{P}((\Id+\Delta t \, \bu+\Delta t \delta\bu) (\Omega)) = \mathcal{P}(\varphi(\Omega)) + \Delta t \int_{\partial \Omega}  \nabla \delta \bu \,\bt\cdot \Tau \D s +O(||\nabla \delta \bu||^2)\, .
	\] 
	Consequently, for $\mathbf{v}: \partial \Omega \to \R^2$ we have
	\begin{equation}
		\langle \mathcal P' \left(\varphi (\Omega)\right); \bv \rangle =\Delta t \int_{\partial \Omega}  \nabla \mathbf{v} \,\bt\cdot \Tau \D s \, .\label{eq:dPer 1} 
	\end{equation}

	Recalling that $s$ is the arc length coordinate of $\partial \Omega$, it follows that $\nabla\mathbf{v}\,\bt=\frac{\D \mathbf{v}}{\D s}$. Then, integrating \eqref{eq:dPer 1} by parts we obtain
	\[
	\langle \mathcal P' \left(\varphi (\Omega)\right); \bv \rangle =- \Delta t \int_{\partial \Omega}  \mathbf{v} \cdot \frac{\D \Tau}{\D s} \D s \, . \label{eq:dPer 1}
	\]
	
	Finally, since 
	$
	\Hup=-\frac{\D \Tau}{\D s}
	$ we deduce \eqref{eq:der_per}.
\end{proof}

\subsection{On functions with bounded variation}

We recall the definition of the total variation for a function $f \in L^1(\Gamma)$:
\[
V(f):=\int_\Gamma |\textrm{D} f| \D s= \sup \left\{ \int_\Gamma f \, g' \D s\, ;\,  g \in  \mathcal C^1(\Gamma,\R^2)\, , \, |g(s)| \le  1 \textrm{ for all }s \in  \Gamma   \right\},
\]
where $f'$ is the derivative of $f$.
The space of functions with bounded variation is then defined as
\[
BV(\Gamma) = \left\{f \in  L^1(\Gamma)\, ;\, 
\int_\Gamma |\textrm{D}  f| \D s< \infty \right\},
\]
and is equipped with the norm
\[
\|f\|_{BV(\Gamma)}=\|f\|_{L^1(\Gamma)}+ \int_\Gamma |\textrm{D} f| \D s.
\]

\begin{rema}
	If $f \in W^{1,1}(\Gamma)$, this coincide with the usual $W^{1,1}(\Gamma)$ norm. However $W^{1,1}(\Gamma)$ is a proper subset of $BV(\Gamma)$, since the derivative of $BV$ functions are in general measures not functions in $L^1(\Gamma)$.
\end{rema}

We recall the following classical results for BV functions.
\begin{theo}
	Let $(\bu_k)_{k\in \N}$ be a bounded sequence in $BV(\Gamma)$. Then there exists a subsequence which converges strongly in $L^1(\Gamma)$.
\end{theo}

A crucial lemma in the study of BV functions is the following approximation result:
\begin{lemm}[(Approximation by $\mathcal C^\infty$ functions)]
	Let $\bu \in  BV(\Gamma)$. There exists a sequence $\bu_k \in  \mathcal C^\infty(\Gamma)$ such that $\bu_k \to  \bu$ in $L^1(\Gamma)$ and
	\[
	\int_\Gamma |\bu_k'| \D s \to  \int_\Gamma |\textrm{D} \bu_k| \D s.
	\]
\end{lemm} 
\begin{rema}
	Note that we cannot expect to have
	\[
	\int_\Gamma |\bu_k'- \bu '| \D s \to  0,
	\]
	in general since that would imply that $\bu \in W^{1,1}(\Gamma)$.
\end{rema}

\section{Explicit time discretization of  \eqref{eq:kin} -- \eqref{eq:mbv}}\label{sec:explicit}

From the time-discrete equations \eqref{eq:mbv discrete_explicit} -- \eqref{eq:discrete propagation_boundary_2}, we derive the variational formulation by standard techniques.

Assuming that $\Omega$ is smooth enough, say $C^2$ e.g., the varitional formulation associated with \eqref{eq:mbv discrete_explicit} is: 
\begin{equation}\label{eq:FV}
	\begin{cases}
		\textrm{find } \bu\in \left(L^2(\Omega)\right)^2 \textrm{ and } p\in H^1 (\Omega) \textrm{ that solves} \\
		\int_{\Omega}\bu\cdot\mathbf{v} \, \dx +\int_{\Omega} \nabla p\cdot \mathbf{v} \, \dx 
		+ \sigma \int_{\Gamma}
		\bH^{i}
		\cdot \mathbf{v} \, \D s =0\, , \\
		  \int_{\Omega}(\nabla\cdot\bu)q \, \dx  =0\, ,\\
		 		\textrm{ for any smooth test functions } \bv:\Omega\to\R^2 \textrm{ s.t. } \Div \bv =0 \\ \textrm{ and } q:\Omega\to\R \textrm{ s.t. } \bv_{\Gamma}=0\, . 		 
	\end{cases}
\end{equation}

Since 
\[
\frac{1}{\Delta t } \mathcal P \left((\Id+\Delta t \, \bu)(\Omega)\right) = \frac{1}{\Delta t }  \left( \mathcal P \left(\Omega\right) +\Delta t \langle  \mathcal P' \left(\Omega\right); \bu \rangle\right), 
\]

We start with an existence result for the minimization of  \eqref{eq:functional}. Recall that the space $V(\Omega)$ is defined by \eqref{def:V}.

\begin{lemm}
	There exists a unique solution of the minimization problem
	\[
	\inf _{\bu \in V(\Omega)} J_{\rm explicit}(\bu).
	\] 
\end{lemm}
\begin{proof}
	The proof of this result is classical. The term $\frac{1 }{2}\int_\Omega |\bu|^2 \D x$ is coercive on $\left( L^2(\Omega)\right)^2$. The set $\{\Div \, \bu =0\}$ is closed in $\left( L^2(\Omega)\right)^2$ weak. Finally,  since $\nabla \kappa $ belongs to $\left( L^2(\Omega)\right)^2$, the linear part of $J_{\rm explicit}$ is continuous.  
\end{proof}

\begin{prop}\label{prop:liens_edp_min}
	$(\bu,p)\in \left(L^2(\Omega)\right)^2 \times H^1(\Omega)$ is solution of \eqref{eq:mbv discrete_explicit} iff 
	\[
	\bu = \argmin_{\bv\in V(\Omega)} J(\bv)\, .
	\]
\end{prop}

\begin{proof}
	Let $f \in \left(L^2(\Omega)\right) ^2$. Consider first the problem with homogeneous Dirichlet condition
	\begin{equation}\label{eq:HS_Dirichlet_hom_second_membre}
		\begin{cases} 
			\bu+\nabla p=f & \text{in }\Omega \, , \\
			\nabla\cdot\bu=0 & \text{in }\Omega \, ,\\
			p=0& \text{on }\Gamma\, .
		\end{cases}
	\end{equation}
	Let 
	\begin{equation}\label{eq:functional_0}
	I(\bu)= \frac{1 }{2}\int_\Omega |\bu|^2 \D x - \int_\Omega f\bu \D x \, .
	\end{equation}

	\begin{lemm}\label{lem:hom_second_membre}
		$(\bu,p)\in \left(L^2(\Omega)\right) ^2\times H^1_0(\Omega)$ is solution of \eqref{eq:HS_Dirichlet_hom_second_membre} iff 
		\[
		\bu = \argmin_{\bv\in V(\Omega)} I(\bv)\, .
		\]
	\end{lemm}
	\begin{proof}
	Let $(\bu,p) \in \left(L^2(\Omega)\right)^2\times H^1_0(\Omega)$ a solution of \eqref{eq:HS_Dirichlet_hom_second_membre}. For all $\delta \bu \in \left(L^2(\Omega)\right)^2$ such that $\Div \, \delta \bu =0$, we have
	\begin{eqnarray*}
		\langle I'(\bu),\delta \bu \rangle &=& \int_\Omega \bu \cdot  \delta \bu \dx - \int_{ \Omega} f \delta \bu \dx		\\
		&=&-\int_\Omega \nabla p \cdot  \delta \bu  \dx	\\
		&=&\int_\Omega p \, \Div \left(\delta \bu \right) \dx\\
		&=& 0\, ,	
	\end{eqnarray*}
	hence $\bu = \argmin_{\bv\in V(\Omega)} I(\bv)$.

	Conversely, assume that $\bu = \argmin_{\bv\in V(\Omega)} I(\bv)$. Let 
	\[
	L(\bv):=\langle I'(\bu),\bv \rangle \, ,
	\] 
	then 
	$L(\bv)=0$ for all $\bv \in  \left(L^2(\Omega)\right)^2$ such that $\Div \, \bv =0$.
	We can apply De Rham's theorem (see theorem \ref{th:deRahm}) to $L: \left(L^2(\Omega)\right)^2 \to \left(L^2(\Omega)\right)^2$ and we get that there exists $\tilde p\in H_0^1(\Omega)$ such that 
	\[
	L(\bv )=\int_\Omega \nabla \tilde p \cdot \bv \,\dx\, .
	\]
	Integrating by parts it comes
	\[
	L(\bv )=-\int_\Omega \tilde p \, \Div \bv \, \dx\, .
	\]
	Setting $p=-\tilde p$, we deduce that $(\bu,p) \in \left(L^2(\Omega)\right)^2 \times H^1_0(\Omega)$ is a solution of \eqref{eq:HS_Dirichlet_hom_second_membre}. 
	\end{proof}
	
	We now apply Lemma \ref{lem:hom_second_membre} to the particular case where $f=-\sigma \nabla \kappa$. Indeed $(\bu,p)$ is a solution of \eqref{eq:HS_Dirichlet_hom_second_membre} iff $(\bu,p+\sigma \kappa)$ is a solution of \eqref{eq:mbv discrete_explicit}.
	
\end{proof}

\begin{lemm}
	The variational formulation \eqref{eq:FV} is the Euler equation for the minimization problem:
	\[
	\inf_{\bu\in V(\Omega)} J_{\rm explicit}(\bu)\, ,
	\]
	where $\Omega$ is a given smooth subset of $\R^2$ ($C^2$ e.g.) and the functional $J_{\rm explicit}$ is defined by \eqref{eq:functional} over the functional space $V(\Omega)$ defined by \eqref{def:V}.
\end{lemm}
\begin{proof}
	Define the lagrangian $\mathcal L$ on $\left(L^2\left(\Omega\right)\right) ^2\times H^1\left(\Omega\right) $ by
	\[
	\mathcal L (\bu, p) = \frac12 \int_{\Omega} |\bu |^2 \,\dx+  \int_{\Omega} \bu \cdot \nabla p \,\dx+ \sigma \int_{ \Omega} \nabla \kappa \cdot \bu \,\dx\, . 
	\]

	The functional $\mathcal L$ is G\^ateaux differentiable with respect to $(\bu,p)$.  Let 
	\[
	\bu = \argmin_{\bv\in V(\Omega)} J_{\rm explicit}(\bv)\, .
	\]
	
	Using the arguments given in the proof of Proposition \ref{prop:liens_edp_min}, we deduce that
	for all $\bv\in V(\Omega)$ and $q\in H^1_0(\Omega)$, we have
	\[
	\langle \mathcal L'(\bu,p);(\bv,q)\rangle
	=\int_\Omega \bu\cdot \bv \, \dx +  \int_{\Omega} \bv \cdot \nabla p \dx+ \sigma  \int_{\Omega} \nabla \kappa \cdot \bv \, \D s +  \int_{\Omega} \bu \cdot \nabla q \dx=0	\, .
	\]
	Integrating by parts we obtain
	\[
	\int_\Omega \bu\cdot \bv \, \D x +  \int_{\Omega}  \bv\cdot \nabla p \, \D x  + \sigma \int_{\partial \Omega} \bH\cdot \bv \, \D s -  \int_{\Omega} \left( \Div \bu\right) q \, \D x=0	\, ,
	\]
	that is \eqref{eq:FV}. 
\end{proof}

\begin{rema}
Recalling that the vector curvature $\bH:= \kappa \bn$ of $\Gamma $ is defined by
\[
\bH=-|\mathbf{X}^{'}(s)|^{-1}\frac{\D }{\D s}\bt(\mathbf{X}(s)) =- \nabla \bt \, \bt\, ,
\]
and integrating by parts on the boundary term in \eqref{eq:FV}, we get
\begin{equation}\label{eq:int_partie_courbure}
	\int_{\Gamma^{i}} \bH^{i} \cdot \mathbf{v}= \int_{\Gamma^{i}}  \bt \cdot \nabla \left(\bt \bv\right) \, .
\end{equation}
When $s$ is the arc length coordinate of $\Gamma$, it comes $\nabla\mathbf{v}\,\bt=\frac{\D \mathbf{v}}{\D s}$. 
Hence, the explicit scheme is very simple to implement but like all explicit schemes it requires a very small time step, see Section \ref{sec:verification}.
\end{rema}

\section{Study of \eqref{pb:well-posed} -- \eqref{pb:well-posed_2}}\label{sec:var_2}

In this part we study  the minimization problem \eqref{pb:well-posed} -- \eqref{pb:well-posed_2}. We start with its origin and the existence of a solution and then we study the variational formulation and its implementation.

\subsection{Origin of  \eqref{pb:well-posed} -- \eqref{pb:well-posed_2}}
The incompressibility constraint 
\[
\nabla\cdot \bu=0 \textrm{ in } \Omega,
\]
can be rewritten as
\[
\int_\Gamma \bu\cdot \bn \D s=0.
\]

Moreover, consider functions such that
\[
\int_\Gamma \psi \cdot \bn \D s=0,
\]
we can define a norm $\|\cdot\|_{H^{-1/2}(\Gamma)}$ on the space $H^{-1/2}(\Gamma)$  by
\[
\| \bpsi \cdot \bn  \|_{H^{-1/2}(\Gamma)}:= \inf_{\tiny \begin{array}{l} \nabla \cdot \bu=0 \\ \bu\cdot \bn=\psi \cdot \bn \end{array}} \|\bu\|_{L^2(\Omega)}.
\]

Consider the minimization problem \eqref{pb:well-posed} -- \eqref{pb:well-posed_2} that we recall
\begin{equation}\label{pb:well-posed_bis}
\inf_{\psi\in BV(\Gamma)} J_{\rm mod}(\psi),
\end{equation}
with
\[
J_{\rm mod}(\psi) :=\|\psi\cdot \bn \|_{H^{-1/2}(\Gamma)} + V\left((\Id +\Delta t \, \psi)(x)\right),
\]
where $V(\cdot)$ is the total variation and $BV(\Gamma)$ is the space of functions with bounded variations.

\begin{lemm}
	There exists a solution to \eqref{pb:well-posed}.
\end{lemm}

\begin{proof}
	Consider a minimizing sequence $(\psi_k)_{k\in \N}$ of the minimization problem \eqref{pb:well-posed}. We first note that $(\psi_k)_{k\in \N}$ is bounded in $BV$, hence, up to a subsequence,  $(\psi_k)_k$ converges to a limit $\psi \in BV$, and we have for almsot every point of the boundary  
	\[
	\lim_{k \to \infty} (\psi_k  \cdot \bn_{|\Gamma}) = \psi\cdot \bn_{|\Gamma},
	\]
	hence the result.
\end{proof}

Furthermore, we can perform an analysis on the tangential component of the solution of \eqref{pb:well-posed_bis}. In this case, it is obvious that the normal component of $\bu$ will be small and assuming that the solution behaves correctly, the tangential component of $\bu$ should also be small.

For simplicity let us write $\bu_\bt:= \bu \cdot \bt$, $\bu_\bn:= \bu \cdot \bn$ and $\dot f:= \frac{\D}{\D s} f$. 

\begin{lemm}
	Assume that $\Omega$ is smooth enough, $\mathcal C^2$ for example. Let $\bu \in BV(\Gamma)$ be such that
	\[
	J_{\rm mod}(\bu)=\inf_{\psi\in BV(\Gamma)} J_{\rm mod}(\psi).
	\]	
	Assume that $\bu_\bt$ and $\bu_n$ are small. 
	Then, almost everywhere on $\Gamma$, we have 
	\[ \kappa=0 \qquad \textrm{ or } \qquad 
	\bu_\bt =  \frac{1}{\sigma \kappa} \frac{\D}{\D s} \bu_\bn.
	\]
\end{lemm}
\begin{proof}
	Define $P$ by 
	\[
	P(\psi)= \int_\Gamma |\bt+\dot \psi| \D s ,
	\]
	and consider the minimization problem
	\[
	\inf_{\psi\in BV(\Gamma)} P(\psi).
	\]
	For all $\phi=\phi_\tau \tau$, we have
	\[
	<P'(\psi),\phi>
	=
	\int_\Gamma \frac{\bt+\dot\psi}{|\bt+\dot\psi|}\cdot\dot\phi \D s.
	\]
	Integating by parts, we get
	\[
	<P'(\psi),\phi>
	=-
	\int_\Gamma \frac{\D}{\D s}\left(\frac{\bt+\dot\psi}{|\bt+\dot\psi|}\right)\cdot \phi \D s=0
	\]
	Hence, almost evrywhere on $\Gamma$, we deduce that
	\[
	\frac{\D}{\D s}\left(\frac{\bt+\dot\psi}{|\bt+\dot\psi|}\right)\cdot\bt = 0.
	\]
	Therefore, there exists a funtion $\beta:\Gamma \rightarrow \R$ such that
	\[
	\left(\frac{\bt+\dot\psi}{|\bt+\dot\psi|}\right) = \beta \bt,
	\]
	thus,
	\[
	\dot\psi = (\beta|\bt+\dot\psi| -1 ) \bt.
	\]
	Consequently $\dot \psi$ is colinear to $\bt$, that is $\dot\psi \cdot \bn =0$.
	
	Since
	\[
	\psi=\bu_\bt \bt + \bu_\bn \bn,
	\]
	we see that
	\[
	\dot \psi = \dot \bu_\bt \bt - \bu_\bt \kappa \bn + \dot \bu_\bn \bn + \kappa \bu_\bn \bt.
	\]
	
	Recalling that the normal composant of $\dot \psi$ is zero, we obtain
	\[
	-\kappa \bu_\bt+\dot \bu_\bn=0.
	\]
\end{proof}

\begin{rema}
	It is not surprising to have no information on the tangential component of the deformation on the regions where the curvature of $\Gamma$ is zero.
	Indeed, in these regions, displacement along the tangent simply corresponds to a reparameterization whose cost function is independent.
\end{rema}	

\begin{rema}
	Using \eqref{def:perm_def}, considering that $s$ is the arc-length and denoting $\dot \bu:=\frac{\D }{\D s} \bu$, we see that $J_0(\Omega, \bu)$ defined by \eqref{eq:func_J_alpha}, with $\alpha =0$, rewrites as
	\[
	J(\Omega, \bu)=
	\frac12 \int_\Omega |\bu|^2 \, \D x
	+
	\frac{\sigma}{\Delta t} \int_\Gamma |\bt+\dot \bu| \, \D s.
	\]
	Consider a minimizing sequence $(\bu_k)_{k\in \N}$ of the minimization problem \eqref{eq:Otto_2}, that we recall
	\begin{equation}\label{pb:ill-posed}
		\min_{\tiny \begin{array}{l}\bu \in L^2(\Omega)\\ \nabla \cdot \bu=0 \end{array}} J(\Omega,\bu).
	\end{equation}
	We first note that $(\bu_k)_{k\in \N}$ is bounded in $H(\operatorname{div})$. Hence, up to a subsequence, the sequence $(\bu_k)_k$ converges to a limit $\bu \in H(\operatorname{div})$, and we have for almsot every point of the boundary  
	\[
	\lim_{k \to \infty} (\bu_k  \cdot \bn_{|\Gamma}) = \bu\cdot \bn_{|\Gamma},
	\]
	where $\bn$ is the normal to the boundary. 
	
	However, without any control on the curl of $\bu$,
	there is no reason that the sequence $(\bu_k\cdot \bt _{|\Gamma})$ converges towards $\bu \cdot \bt _{|\Gamma}$. 
	Indeed as we will in Section \ref{sec:verification}, even with a frequent remeshing, we obtain a shift between the tangential displacement of the internal nodes and the boundary nodes. 
	For this reasons, it is difficult to use the original formulation numerically and we propose a second method. 
\end{rema}

As a conclusion, the perimeter minimization problem on zero divergence fields \eqref{pb:ill-posed} is ill-posed, but correctly reformulated by \eqref{pb:well-posed_bis} on the domain boundary. For \eqref{pb:well-posed_bis} and for small displacements, the tangential component depends on the derivative of the normal component. We therefore expect a loss of regularity of the tangential component with respect to the normal component outside the parts of the boundary where $\kappa$ is zero. We cannot expect any regularity (other than bounded variation) of the tangential component in the areas where $\kappa=0$. Finally, in case of discretization of the initial problem \eqref{pb:ill-posed}, one should see boundary layers appearing on the tangential component of $\bu$ whose thickness depends on the mesh used. We will see some situations where it is not the case in Section \ref{sec:verification}.

\subsection{Variational formulation associated with \eqref{pb:well-posed} -- \eqref{pb:well-posed_2}}

Define the functional space $W(\Omega)$ by 
\begin{eqnarray*}
W(\Omega)&=&\Big\{\left(\bU, \bu\right)\in \left(L^2(\Omega)\right)^2 \times \left( H^{-\frac12}\left( \Gamma\right) \cap BV\left( \Gamma\right)\right) \\
& &
\qquad \textrm{ s.t. for all } p\in H^1(\Omega) \, \int_\Gamma p \bu \cdot \bn \, \D s=\int_\Omega \bU \cdot \nabla p \, \dx     \Big\},
\end{eqnarray*}
and the functional $J^*_{\rm mod}$ by
	\begin{equation}\label{pb:well-posed_mod_2}
	J^*_{\rm mod}\left(\bU, \bu\right) :=\frac12\|\bU \|^2_{L^{2}(\Omega)} + \int_{\Gamma}\left|\bt  +\Delta t \, \dot{\bu}\right|\, \D s .
\end{equation}

\begin{prop}
	The minimization problem \eqref{pb:well-posed} -- \eqref{pb:well-posed_2} is equivalent to the following one:
	\begin{equation}\label{pb:well-posed_mod}
		\inf_{\left(\bU, \bu\right)\in W(\Omega)} J^*_{\rm mod}\left(\bU, \bu\right).
	\end{equation}
\end{prop}

\begin{proof}
	The proof relies on the following equivalence:
	\[
	\int_\Gamma \bu \cdot \bn \, \D s = 0,
	\]
	is equivalent to
	\[
	\begin{cases}
	\textrm{there exists }\bU \in \left( L^2(\Omega)\right) ^2 \textrm{ such that } \\	
	\Div \bU = 0 \textrm{ in } \Omega \textrm{ and  } \bU \cdot \bn = \bu \cdot \bn \textrm{ almost everywhere on } \Gamma.
	\end{cases}
	\]
	
\end{proof}

Define the variational formulation:
\begin{equation}\label{eq:FV_well-posed}
	\begin{cases}
		\textrm{find } \left(\bU, \bu\right)\in \left(L^2(\Omega)\right)^2 \times \left( H^{-\frac12}\left( \Gamma\right) \cap BV\left( \Gamma\right)\right) \\
		\textrm{and } p\in H^1 (\Omega) \textrm{ that solve} \\
		\int_{\Omega}\bU\cdot\mathbf{V} \, \dx +\int_{\Omega} \nabla p \cdot \mathbf{V} \, \dx 
		+ \sigma \int_{\Gamma}
		\widetilde{\bH}
		\cdot \mathbf{v} \, \D s =0\, , \\
		\int_{\Gamma} q \, \bu\cdot \bn  \, \D s =0\, ,\\
		\textrm{ for any smooth test functions } \mathbf{V}:\Omega\to\R^2 \, , \bv \in BV(\Gamma ^i) \\ \textrm{ and } q:\Omega^i\to\R \textrm{ s.t. } \bv_{\Gamma}=0\, . 		 
	\end{cases}
\end{equation}

\begin{prop}
	The variational formulation \eqref{eq:FV_well-posed} is the Euler equation for the minimization problem \eqref{pb:well-posed} -- \eqref{pb:well-posed_2}.
\end{prop}	

\begin{proof}
		Define the lagrangian $\mathcal L_{\rm mod}$ by
	\[
	\mathcal L_{\rm mod} (\bu,\bU, p) = \frac12 \int_{\Omega} |\bU |^2 \, \dx + \int_\Gamma \left( \bu \cdot \bn \right) p \, \D s-  \int_{\Omega} \bU \cdot \nabla  p \, \dx+ \sigma \mathcal{P} ((\Id+\Delta t \, \bu)(\Omega))\, . 
	\]

	Let $F(\bu)$ be the perimeter functional in the deformed configuration, i.e.,
	\begin{equation}
		F(\bu) = \mathcal{P} ((\Id+\Delta t \, \bu)(\Omega))=\int_\Gamma \left|(\bI+\Delta t \nabla \bu) \bt\right| \, \D s\, .\label{eq:perimeter functional}
	\end{equation}
	Then, 
	\[
	F'(\bu)(\mathbf{v})=
	\mathcal{P}'\left((\Id + \Delta t \, \bu)(\Omega);\mathbf{v}\right) \, ,
	\]
	and using Lemma \ref{lem:shape_der_per}, we compute the first order expansion of F at $\bu$.

	\begin{lemm}
		It holds that 
		\begin{equation}
			\int_{\Gamma} \widetilde{\bH}\cdot\bv \, \D s= \frac{1}{\Delta t} F'(\bu)(\mathbf{v})
			=  \int_{\Gamma^i}\nabla \mathbf{v}\,\bt \cdot \Tau \, \D s .
			\label{eq:Implicit shape derivative_0}    
		\end{equation}
	\end{lemm}
The result then follows.
\end{proof}		
	\begin{rema}
		Note that the expression on the right-hand side of \eqref{eq:Implicit shape derivative_0} absolves us from the difficult task of computing $\widetilde{\bH}$ directly. That being said, our problem is still nonlinear through the dependence of $\Tau$ on $\bu$, see \eqref{eq:Tau i+1}.
	\end{rema}

	\subsection{A Newton algorithm to solve \eqref{eq:FV_well-posed} } \label{Newton method}
	
	The main difficulty in solving \eqref{eq:FV_well-posed} is to find a method to handle the nonlinear boundary term,  \eqref{eq:Implicit shape derivative_0}. 
	
	We solve the problem \eqref{eq:FV_well-posed} by a Newton method by seeking, for $\bu$ given, a correction $\delta \bu $ such that $(\delta \bu,p)$ is solution of the linearized system
	\begin{equation}
		\begin{aligned}
			\int_{\Omega}\left(\bu +\delta \bu \right)\cdot\mathbf{v} \, \dx-\int_{\Omega} p(\nabla \cdot\mathbf{v}) \, \dx 
			&- \int_{\Omega}\left(\nabla\cdot\left(\bu +\delta \bu \right)\right) q \, \dx
			\\
			+ & \frac{\sigma}{\Delta t} F'(\bu) (\mathbf{v}) +\frac{\sigma}{\Delta t}  F''(\bu) \left(\delta \bu,\mathbf{v}\right)=0 \, .
		\end{aligned}, \label{eq:var linearization}
	\end{equation}
	
	\begin{rema}
		The last term in Eq. \eqref{eq:var linearization} is a bilinear form. It is obtained by finding the second order asymptotic expansion of $F$ at $\bu$.
	\end{rema}
	
	More precisely, we proceed iteratively. At each time step, it consists of computing a sequence $(u^k)_k$ where $\bu^0=0$ and $ \bu^{k+1} \in \left(H^1(\Omega)\right) ^2 $, $p \in L^2(\Omega)$ are solutions to the following variational problem
	\begin{equation}
		\begin{aligned}
			\int_{\Omega}\bu^{k+1}\cdot\mathbf{v} \, \dx&-\int_{\Omega} p(\nabla \cdot\mathbf{v}) \, \dx 
			- \int_{\Omega}(\nabla\cdot\bu^{k+1})q \, \dx
			\\
			+ &\frac{\sigma}{\Delta t} F'(\bu^k)(\mathbf{v})+ \frac{\sigma}{\Delta t}  F''(\bu^k) \left(\bu^{k+1}-\bu^k,\mathbf{v}\right)  =0 
		\end{aligned}    \label{eq:var droplet linearized}
	\end{equation}
	for any arbitrary smooth test functions $\mathbf{v}$, $q$.
	
	Assuming that this method converges, i.e., 
	$
	\lim_{k\to\infty} |\bu^{k+1}-\bu^k|\to 0
	$, 
	we shall denote by $\bu$ the limit of $(u^k)_k$ for $k\to\infty$. The remaining challenge is to compute $F''$.
	
	\subsubsection*{Second order expansion of the deformed perimeter}
	To apply our Newton like method we must first perform an asymptotic expansion of order two of the perimeter functional in the deformed configuration.

	Let us compute the second order Taylor expansion of $F$ defined by \eqref{eq:perimeter functional}.
	
	Let $\D s$, $\bt$ and $\bn$ denote respectively the unit length, the tangent and the outward pointing unit normal vectors in the current configuration $\Omega$, and $\D S$, $\Tau$ and $\N$ are respectively the unit length, the tangent and the outward pointing unit normal vectors in the deformed configuration $(\Id + \Delta t   \bu)(\Omega)$, "pulled back" in the coordinate system of the current configuration.

	\begin{lemm}
		It holds that
		\begin{equation}
			F''(\bu)(\delta \bu,\mathbf{v})=\Delta t^2 \int_\Gamma \frac{ \left(\nabla \delta\bu\,\bt\cdot\N \right)\left(\nabla \mathbf{v}\,\bt\cdot\N \right)}{\D S} \D s\, .\label{eq:F(u) 2nd order}
		\end{equation}
	\end{lemm}
	\begin{proof}
		We see that
		\begin{equation}
			\begin{aligned}
				F(\bu+\delta \bu)=&F(\bu) + \Delta t \int_\Gamma \nabla \delta \bu \, \bt \cdot \Tau \D s \\
				&+ \frac{\Delta t^2}{2} \int_\Gamma \frac{ \left(\nabla \delta\bu\,\bt\cdot\N \right)^2}{\D S} \D s + O(||\nabla \delta \bu||^3) \, .
			\end{aligned} \label{eq:F(u) 2nd order}
		\end{equation}	
	\end{proof}
	
	\subsubsection*{First linearized problem}
	Substituting \eqref{eq:Implicit shape derivative_0} and \eqref{eq:F(u) 2nd order} into our Newton like method, \eqref{eq:var droplet linearized}, gives
	\begin{equation}
		\begin{aligned}
			&\int_{\Omega}\bu^{k+1}\cdot\mathbf{v} \, \dx-\int_{\Omega} p(\nabla \cdot\mathbf{v}) \, \dx 
			-  \int_{\Omega}(\nabla\cdot\bu^{k+1})q \, \dx
			\\
			& + \sigma \int_\Gamma \nabla \mathbf{v} \, \bt \cdot \Tau^k \D s+ \sigma\Delta t \int_\Gamma \frac{(\nabla \delta \bu^{k+1}\, \bt\cdot \N^k)(\nabla \mathbf{v}\, \bt\cdot \N^k)}{\D S^k} \D s  =0 
		\end{aligned}    \label{eq:First linearized}
	\end{equation}
	where $\delta\bu^{k+1}:=(\bu^{k+1}-\bu^{k})$ and
	\[
	\D S^k=|(\I +\Delta t\, \nabla \bu^k)\bt|, \quad \Tau^k=\frac{(\I +\Delta t\, \nabla \bu^k)\bt}{|(\I +\Delta t\, \nabla \bu^k)\bt|},\quad \N^k=\Tau^{k^\perp}\, .
	\]
	
	\begin{rema}
		We emphasize the fact that \eqref{eq:First linearized} does not contain second-order spatial derivatives, meaning that the curvature vector has been transformed into a term involving only the first spatial derivatives (hence less regularity is needed).
	\end{rema}
	
	This system in the variables ($ \bu^{k+1}$, $p$) 
	does not necessarily admit a solution (let alone unique) due to the possible lack of coercivity of the bilinear form. Moreover, even if solutions to \eqref{eq:First linearized} exist, the convergence of this iterative Newton method, \eqref{eq:var droplet linearized} with \eqref{eq:First linearized}, is not granted. The straighforward remedy is to replace the bilinear term in \eqref{eq:First linearized} by a coercive one, which can be done in various ways. The classical scheme is to compute the eigenvectors and eigenvalues of the matrix and to construct the modified term using only the eigenspaces associated with the positive eignevalues. 
	
	For simplicity and robustness of our algorithm, we prefer to predetermine a closed formula for the modified matrix which is coercive. 
	
	\subsubsection*{Modified problem}
	The main obstacle in our current method is that the bilinear form in \eqref{eq:First linearized} is not positive definite. Therefore, we want to define a modified problem of \eqref{eq:var droplet linearized} -- \eqref{eq:First linearized} with a positive-definite matrix.

	\begin{lemm}
		It holds that
		\begin{equation}\label{eq:F(u) 2nd order_2}
			F(\bu+\delta \bu)\leq  F(\bu)+ \Delta t  \int_\Gamma  \nabla \delta \bu\,\bt \cdot \Tau \D s  + \frac{\Delta t^2}{2} \int_\Gamma \frac{(\nabla \delta \bu \,\bt )\cdot (\nabla \delta \bu \,\bt )}{\D S}\D s\, .
		\end{equation}
	\end{lemm}
	\begin{proof}
		From the inequality $\sqrt{1+x}\leq 1+x/2$ it follows  that
		\begin{equation}
			\begin{aligned}
				&|(\I+\Delta t \, \nabla(\bu+\delta\bu))\bt|\\
				& \quad =
				\Big( |(\I+\Delta t\, \nabla\bu)\bt|^2+2\Delta t \nabla \delta \bu\, \bt \cdot (\bt+\Delta t \, \nabla\bu\, \bt) +\Delta t^2 |\nabla \delta \bu\, \bt |^2\Big)^{1/2}\\
				&\quad \leq\, |(\I+\Delta t \, \nabla\bu)\bt|+\Delta t \,\nabla\delta\bu \, \bt \cdot \Tau + \frac{\Delta t^2 |\nabla\delta\bu\,\bt|^2 }{2 \D S} \, ,
			\end{aligned}
		\end{equation}
		where $\D S$ and $\Tau$ are defined by \eqref{eq:First linearized}.
	\end{proof}

	We stress that here, unlike \eqref{eq:F(u) 2nd order}, the bilinear form on the right-hand side of \eqref{eq:F(u) 2nd order_2} is positive definite. We choose to adopt this term in our modified Newton like method. 
	
	To recapitulate, in each time step we omit the index $i$ and compute a sequence $(u^k)_k$, where $\bu^0=0$ and $ \bu^{k+1} \in H^1(\Omega)^2 $, $p \in L^2(\Omega)$ are solutions of the following variational problem 
	\begin{equation}
		\begin{aligned}
			\int_{\Omega}\bu^{k+1}\cdot\mathbf{v} \, \dx&-\int_{\Omega} p(\nabla \cdot\mathbf{v}) \, \dx 
			-  \int_{\Omega}(\nabla  \cdot\bu^{k+1})q \, \dx
			\\
			& + \sigma \int_\Gamma \nabla \mathbf{v} \, \bt \cdot \Tau^k \D s+ \sigma\Delta t \int_\Gamma \frac{(\nabla \delta \bu^{k+1}\, \bt) \cdot (\nabla \mathbf{v}\,\bt) }{\D S^k} \D s  =0 
		\end{aligned}    \label{eq:Second linearized}
	\end{equation}
	for all arbitrary smooth test functions $\mathbf{v}:\Omega \to \R^2$, $q:\Omega\to\R$, and where 
	\[
	\D S^k=|(\I +\Delta t\, \nabla \bu^k)\bt|, \quad \Tau^k=\frac{(\I +\Delta t\, \nabla \bu^k)\bt}{|(\I +\Delta t\, \nabla \bu^k)\bt|}
	\]
	
	In our algorithm, the method is applied recursively until the stopping criteria based on the computation of the global residual is satisfied. We set the Newton tolerance to $10^{-5}$ in our computations, i.e., 
	\[
	\int_\Omega  \left| \delta \bu^{k+1}\right|^2=\int_\Omega  \left| \bu^{k+1} - \bu^k\right|^2 <10^{-5}
	\]
	
	When this condition is satisfied, we take ($\bu^{k+1}$, $p$) as the approximate solution for ($\bu^{i+1}$, $p^{i+1}$) in the time-discrete PDE problem. All that is left is to propagate the domain via Eq. \eqref{eq:discrete propagation}.

\section{Variational formulation associated with \eqref{eq:Otto_3} -- \eqref{eq:func_J_alpha}}\label{sec:stab_scheme}

Let $\alpha>0$, define
\begin{equation}\label{def:J_alpha}
J_\alpha(\Omega,\bu)= \frac 1 2 \int_\Omega |\bu|^2 \, \dx+ \frac \sigma{\Delta t} \mathcal P(\varphi(\Omega)) + \frac \alpha 2 \int_\Omega |\nabla\wedge \bu|^2 \, \dx.
\end{equation}
\begin{prop}
	Let $\Omega$ be a $C^2$ open set of $\R^2$. Any minimizer $\bu$ of $J_\alpha(\Omega,\cdot)$ with respect to $\bu$ over the set of zero divergence fields satisfies
	\begin{equation}\label{EDP2}
	\left\{
	\begin{array}{ll}
		\bu-\alpha \nabla^\perp(\nabla\wedge \bu)+\nabla p =0 &\text{ in }\Omega,\\
		\nabla\cdot \bu=0&\text{ in }\Omega,\\
		p\bn=\sigma \Hup+\alpha(\nabla\wedge \bu) \bt&\text{ on }\partial\Omega.
	\end{array}
	\right.
	\end{equation}
\end{prop}

\begin{proof}
		For any test function $v$, we have
		\[
		\int_\Omega \bu\cdot \bv \, \dx + \alpha \int_\Omega (\nabla \wedge \bu)(\nabla \wedge \bv)\, \dx+\sigma \int_{\partial \Omega} \Hup\cdot \bv \, \D s
		=
		-\int_\Omega \nabla p\cdot \bv \, \dx+ \int_{\partial\Omega} p(\bv\cdot \bn) \, \D s.
		\]
		By integration by parts, we see that 
		\[
		\int_\Omega  (\nabla \wedge \bu)(\nabla \wedge \bv)\, \dx=-\int_\Omega \nabla^\perp(\nabla \wedge \bu)\cdot \bv \, \dx + \int_{\partial \Omega} (\nabla\wedge \bu)\bn\wedge \bv \, \D s.
		\]
\end{proof}

Define the variational formulation:
\begin{equation}\label{eq:FV_curl}
	\begin{cases}
		\textrm{find } \bu\in L^2(\Omega) 
		\textrm{and } p\in H^1 (\Omega) \textrm{ that solve} \\
		\int_{\Omega}\bu\cdot\bv \, \dx + \alpha \int_{\Omega} (\nabla\wedge \bu)\cdot (\nabla\wedge \bv)\, \dx  +\int_{\Omega} \nabla p \cdot \bv \, \dx 
		+ \sigma \int_{\Gamma}
		\widetilde{\bH}
		\cdot \mathbf{v} \, \D s =0\, , \\
		\int_{\Gamma} q \, \bu\cdot \bn  \, \D s =0\, ,\\
		\textrm{ for any smooth test functions } \bv:\Omega\to\R^2 \textrm{ and } q:\Omega^i\to\R \, . 		 
	\end{cases}
\end{equation}

\begin{prop}
	The variational formulation \eqref{eq:FV_curl} is the Euler equation for the minimization problem of \eqref{eq:Otto_3} -- \eqref{eq:func_J_alpha}.
\end{prop}	
\begin{proof}
	The functional $J_\alpha$ is G\^ateaux differentiable with respect to $\bu$. Denoting by $J_\alpha'$ its  G\^ateaux derivative and using lemma \ref{lem:shape_der_per}, we obtain for all $\bv\in \left(H^1(\Omega)\right)^2$, 
	\begin{eqnarray}
		\langle J_\alpha'(\Omega,\bu);\bv\rangle
		&=&\int_\Omega \bu\cdot \bv \, \dx+ \sigma \int_{\partial \varphi(\Omega)} \bH_{\varphi}\cdot (\bv\circ \varphi^{-1}) \, \D s_\varphi \nonumber \\
		& & \qquad +  \alpha \int_{\Omega} (\nabla\wedge \bu)\cdot (\nabla\wedge \bv)\, \dx\nonumber \\
		&=& \int_\Omega \bu\cdot \bv \, \dx+ \sigma \int_{\partial\Omega} \Hup \cdot \bv \, \D s	\nonumber \\
		& & \qquad +  \alpha \int_{\Omega} (\nabla\wedge \bu)\cdot (\nabla\wedge \bv)\, \dx\, .\label{derivee}
	\end{eqnarray}

Define the lagrangian $\mathcal L_\alpha ^*$ by
\[
\mathcal L^* (\bu, p) = \frac12 \int_{\Omega} |\bu |^2 \, \dx+ \frac \alpha 2 \int_\Omega |\nabla\wedge \bu|^2 \, \dx-  \int_{\Omega} \left( \Div \bu\right) p \, \dx+ \sigma \int_{\partial \Omega} \Hup \cdot \bu \, \D s\, . 
\]
We proceed as in the proof of Proposition \ref{prop:liens_edp_min} by applying de Rahm's Theorem \ref{th:deRahm} to the linear form $\langle \mathcal L_\alpha ^{*'}(\bu,p);\cdot\rangle $, to deduce that if $\bu$ is a minimizer of $J(\Omega,\cdot)$ over the set of divergence-free fields, then there exists $p$ such that for all Lipschitz test function $\bv$ and $q$,
\begin{equation}\label{deRahm}
	\langle  \mathcal L_\alpha ^{*'}(\bu,p);(\bv,q)\rangle = \int_\Omega p \nabla\cdot \bv \, \dx\, .
\end{equation}
\end{proof}

\begin{rema}
	Note that a priori, it does not seem obvious that the obtained scheme is consistent, due to the presence of the term $\alpha \nabla\wedge \bu$.
	By taking the rotational of the first equation of the system \eqref{EDP2}  and the dot product with the tangent $\tau$ at the boundary for the last equation of \eqref{EDP2}, we get
	\[
	\left\{
	\begin{array}{ll}
		\nabla\wedge \bu-\Delta (\alpha \nabla\wedge \bu) = 0&\text{ in }\Omega\\
		\alpha (\nabla\wedge \bu)=-\sigma \Hup\cdot\bt &\text{ on }\partial\Omega
	\end{array}
	\right.
	\]
	The boundary term is of order $\Delta t$. Indeed, if $\bu$ is smooth, $C^2$ e.g., we have
	\[
	\Hup =\nabla((\Id+\Delta t \nabla \bu)\bt) \cdot \bt = \Delta t \partial^2_s \bu\cdot \bt.
	\]
	In particular we deduce
	\[
	\alpha \|\nabla \wedge \bu\|_{H^1(\Omega)} \leq C(\Omega) \, \Delta t \, \sigma \|\partial_s^2\bu \cdot \bt\|_{H^{1/2}(\partial\Omega)}.
	\]
	The stabilization term therefore introduces an error of order $\Delta t$. Note that the introduced consistency error is independent of $\alpha$.
\end{rema}

\section{Variational formulation associated with \eqref{eq:Otto_4} -- \eqref{espace:incompressible}}\label{sec:nonlinear}

In order to alleviate the problem of the existence of the time-discretized scheme proposed in the previous section, another option consists in replacing the linearized incompressibility constraint $\nabla\cdot \bu=0$ by the nonlinear constraint  $\det(\nabla \varphi)=1$.

\begin{prop}
Let $\bu$ be a minimizer of $J$ under the constraint
\[
\det(\nabla \varphi)=1\text{ in }\Omega,
\]
with $\varphi=\Id+\Delta t \bu$. Then, there exists a Lagrange multiplier $p$ such that 
\[
\left\{
    \begin{array}{rcl}
     \bu+(\Cof \nabla \varphi) \nabla p=0&\text{ in }&\Omega \\
    \det(\nabla \varphi)=1&\text{ in }&\Omega\\
     p\bn =\Hu \circ\varphi&\text{ on }&\partial\Omega.
    \end{array}
\right.
\]
\end{prop}

\begin{proof}
	In a classical way, to take into account the constraint of nonlinear incompressibility, one introduces the Lagrangian
	\begin{equation}\label{def:lagran}
	\mathcal L(\Omega;\bu,p)=J(\Omega;\bu)-\frac{1}{\Delta t}\int_\Omega (\det(\nabla \varphi)-1) p \, \dx.
	\end{equation}
	Define $V=\bv\circ \varphi^{-1}$ and  $P=p\circ \varphi^{-1}$. We first see that
	\[
	\nabla _\by V = \nabla_\bx \bv (\nabla_\bx \varphi)^{-1}\, ,
	\]
	with the notation $\by = \varphi (\bx)$.
	
	Moreover, we compute
	\begin{eqnarray*}
		\int_\Omega \det(\nabla \varphi + \Delta t \nabla  \bv) p \, \dx
		&=&
		\int_{\varphi(\Omega)} \det(\nabla _\bx\varphi)^{-1} \det(\nabla _\bx\varphi + \Delta t \nabla _\bx \bv) P \, \dx\\
		&=&
		\int_{\varphi(\Omega)}\det(\Id + \Delta t \nabla _\by V) P\, \dx\\
		&=&
		\int_{\varphi(\Omega)}(1+\Delta t \nabla_\by\cdot V) P \, \dx + o(\bv).
	\end{eqnarray*}
    
    Therefore, the G\^ateaux derivation of the last term of \eqref{def:lagran} is   
    \[
    -\frac{1}{\Delta t}\int_{\varphi(\Omega)}(\nabla_\by\cdot V) P \, \dx
    ,
    \]
    and integrating by parts we get
    \[
        \frac{1}{\Delta t}\int_{\varphi(\Omega)}(\nabla_\by\cdot V) P \, \dx
    =-\int_{\varphi(\Omega)} V\cdot \nabla_\by P \, \dx+ \int_{\partial \varphi(\Omega)} V\cdot \bn_\varphi P \D s.
    \]
    Finally, recalling that	$P\circ\varphi = p$, it follows that
    \[
       \nabla_\by P \nabla \varphi = \nabla_\bx p,
    \]
    hence
    \[
        \nabla \varphi^T \nabla_\by P = \nabla_\bx p\,
    \]   
    that is 
    \[
    \nabla_\by P=\nabla \varphi^{-T} \nabla_\bx p.
    \]
    
    We thus deduce that the G\^ateaux derivation of the last term of \eqref{def:lagran} is 
    \begin{eqnarray*}
        -\frac{1}{\Delta t}\int_{\varphi(\Omega)}(\nabla_\by\cdot V) P \, \dx
    &=&
        \int_\Omega \bv\cdot \nabla \varphi^{-T} \nabla_\bx p \det(\nabla \varphi) \, \dx
       - \int_{\partial \varphi(\Omega)} V\cdot \bn_\varphi P \, \D s\\
    &=&
         \int_\Omega \bv\cdot (\Cof \nabla \varphi) \nabla_x p \, \dx
        - \int_{\partial \varphi(\Omega)} V\cdot \bn_\varphi P \, \D s\, .
    \end{eqnarray*}

	If $\bu$ is a minimizer of $J$ with respect to $\bu$ under the constraint $\det(\nabla \varphi) = 1$, then $\bu$ is a critical point of the Lagrangian.
	Using compactly supported test functions in $\Omega$, we deduce that
    \[
        \bu+(\Cof \nabla \varphi) \nabla p=0 \quad \text{ in }\Omega\, ,
    \]
    and by using  any test functions in $\Omega$ we obtain the boundary condition
    \[
         p \bn_\varphi=\Hu  \quad \text{ on }\partial \varphi(\Omega)\, ,
    \]
    that is
    \[
        p\bn_\varphi \circ \varphi =\Hu \circ \varphi  \quad \text{ on }\partial\Omega\, .
    \]
\end{proof}

\section{\texorpdfstring{Verification: simulation-theory comparisons}{Verification: simulation-theory comparisons}}
\label{sec:verification}

\subsection{Space discretization}\label{sec:space_discre} 
Each subdomain is covered by a regular triangulation $\mathcal{T}_h$, with maximum mesh size $h$, and such that it is globally a conforming triangulation of $\Omega$, i.e., $\mathcal{T}_h$ contains a piecewise affine approximation $\Gamma_h$ of the interface $\Gamma$. 

In a classical manner, we approximate each component of the velocity in each element $K \in \mathcal{T}_h$ by a polynomial of degree one enriched with a ``bubble'' function (a polynomial of degree 3 defined as the product of the barycentric coordinates in $K$  and vanishing on the faces of $K$) and the pressure in each element by a polynomial of degree one. Both approximations are continuous across the element faces except for the pressure at the interface $\Gamma_h$. Hence, we consider the following discretizations of the spaces $X:=\left(H^1(\Omega)\right)^2$ and $M:=L^2(\Omega)$:
\[X_h=\{v_h\in \mathcal{C}^0(\bar \Omega) ^2\, ;  \ \forall K \in \mathcal{T}_h \, , \ v_{h|K}\in \left( \mathbb{P}_1+b_K\right)^2\}\cap X\, ,\]
and
\[M_h=\{q_h\in \mathcal{C}^0(\bar \Omega)  \, ;  \ \forall K \in \mathcal{T}_h \, , \ q_{h|K}\in  \mathbb{P}_1\}\cap M\, ,\]

The numerical experiments presented in the following sections were conducted using FreeFem++.

In this section, we validate our FEM simulation by comparing numerical experiments with theoretical predictions. 

To simulate the passive droplet problem (using our code in FreeFem++) one must specify the tension $\sigma$,  the numerical time step $\Delta t$ and the initial finite element domain (meaning the triangulation mesh $\mathcal{T}_h$). 
Let us first comment on the choice of the time step size. The convergence of the Newton algorithm is indeed quite sensitive to $\Delta t$. Values which are too large may lead to starting iterations in our recursive method which are far from the expected solutions. 
There exist several strategies for improving the choice of $\Delta t$ but we leave this topic beyond the scope of the present study. 
In each numerical experiment, we choose a value of $\Delta t$ that is 2--4 orders of magnitude smaller than the physical timescale. 
Building the finite element domain is done by inputting an explicit counter-clockwise paremeterization of the closed interface. Here, we use the polar parameterization $\left\{x(\theta),y(\theta)\right\}=\left\{R^0(\theta) \cos(\theta),R^0(\theta)\sin (\theta)\right\}$ for $\theta\in(0,2\pi)$. We define $R^0(\theta)=1+\delta R^0 (\theta)$, where we write $\delta R^0 (\theta) $ in terms of Fourier modes
\begin{equation}
	\delta R^0(\theta)=\sum_m \left(\delta R_{\text{c} m}^0 \cos (m\theta) + \delta R_{\text{s} m}^0 \sin(m\theta) \right) \label{eq:HS shape perturbation}
\end{equation}
Then, by specifying the number (or density) of vertices along the parameterized boundary, FreeFem++ automatically generates the internal triangulation mesh (see Fig. \ref{fig:ver FEM meshing}a,b). Note that one could also define an adaptive (non-uniform) mesh, as in Fig. \ref{fig:ver FEM meshing}c, which is designed to have a finer definition of vertices along the boundary. The motivation behind such a mesh is to improve the resolution of the shape itself while not 'hyper-meshing' the bulk and thereby drastically increasing computation time (of order $h^{-2}$). 

\begin{figure}
	\begin{center}
		\includegraphics[scale=0.65]{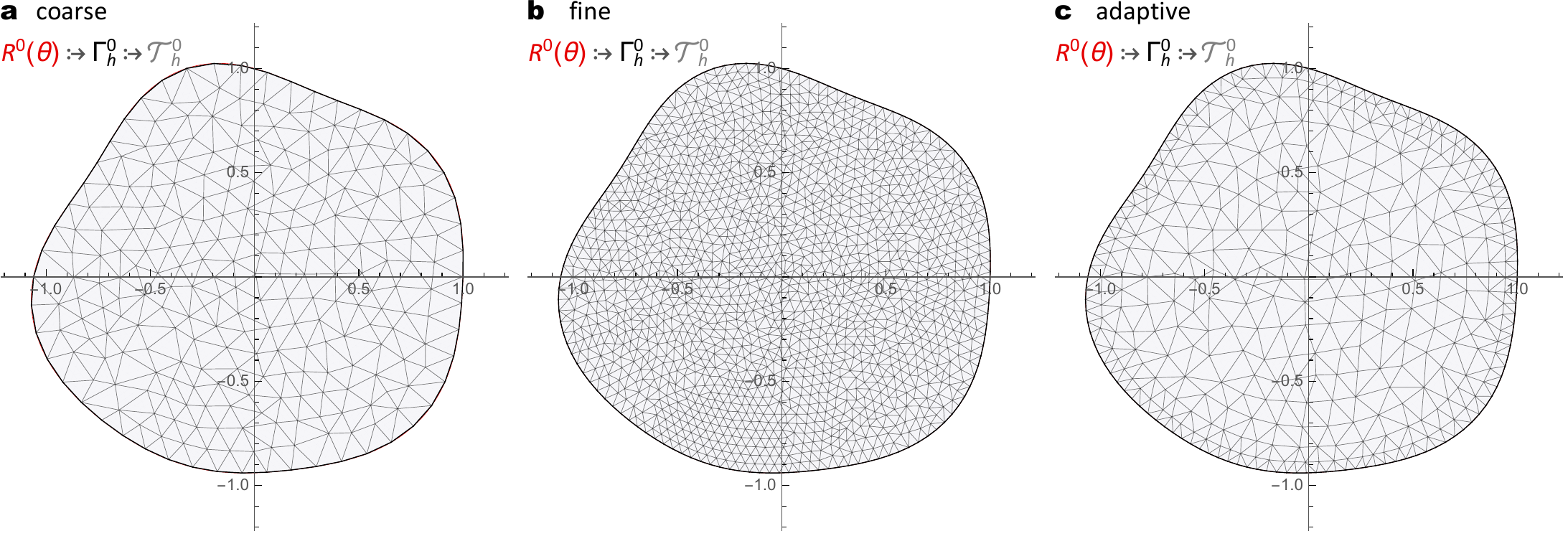}
		\caption{\textbf{Building different finite element domains for the same interface}. The parameterization $R^0(\theta)=1+\delta R^0 (\theta)$ (red, Eq. \eqref{eq:HS shape perturbation}) is used to define the affine approximation of the initial interface, $\Gamma_h^0$ (black). In turn, FreeFem++ uses $\Gamma_h^0$ to generate the internal tiangulation mesh $\mathcal{T}_h^0$ (gray). The density of the mesh is controlled by specifying the number of vertices on the interface. This number is increased from (a) to (b). The adaptive mesh (c) is constructed with two starting interfaces: one (dense) interface on the boundary and another (coarse) interface in the bulk. These define two triangulation meshes (a "disk" and a "ring") which are then joined to form a single connected mesh. \label{fig:ver FEM meshing}}
	\end{center}
\end{figure}

In each time step, we follow the algorithm outlined at the end of Section \ref{Newton method} (Eq. \eqref{eq:Second linearized}). 
The simulation data is saved at some fixed interval of time iterations (of order 10-100, depending on $\Delta t$ and the duration of the simulation). In each imported frame, we are able to reconstruct the finite-element domain and the interpolation functions for $\bu$ and $p$. In addition, we use the interface vertices to construct a polar piece-wise interpolation function of the boundary, $R^\text{sim}(\theta)$. 

Our objective is to square the simulation results with known characteristics of the passive droplet, namely:
\begin{itemize}
	\item Conservation of the droplet area, $\dot{A}=0$. 
	\item External force balance, 
	$\bu_\text{cm}=0$. 
	\item Morphological relaxation dynamics of linear shape perturbations. 
\end{itemize}
\begin{rema}
	There are several ways of computing $\bu_\text{cm}$. We define 
	\[
	\bu_\text{cm}=A^{-1}\int_{\partial\Omega} \mathbf{x}(\bu\cdot\bn)\D l.
	\]
\end{rema}
\begin{rema}
	To find the numerical growth rate of each normal mode, we first decompose $R^\text{sim}(\theta)$ into Fourier components,
	\begin{equation}
		\begin{aligned}
			& \delta R^\text{sim}_{\text{c}m}=\frac{1}{\pi}\int_0^{2\pi}\left( R^\text{sim}(\theta)-1 \right)\cos (m\theta)d\theta \\
			&\delta R^\text{sim}_{\text{s}m}=\frac{1}{\pi}\int_0^{2\pi}\left( R^\text{sim}(\theta)-1 \right)\sin (m\theta)d\theta
		\end{aligned}
		\label{eq:sim Fourier components}
	\end{equation}
	The growth rate of each such component is then obtained by fitting $\delta R_{\text{c}m}^\text{sim}(t)$ and $\delta R_{\text{s}m}^\text{sim}(t)$ to an exponential function $\varepsilon e^{s_m^\text{N} t}$. The idea is to compare the fitted $s_m^\text{N}$ with the classical cubic dispersion relation, $s_m=-\sigma m(m^2-1)$. 
\end{rema}

\begin{figure}
	\begin{center}
		\includegraphics[scale=0.6]{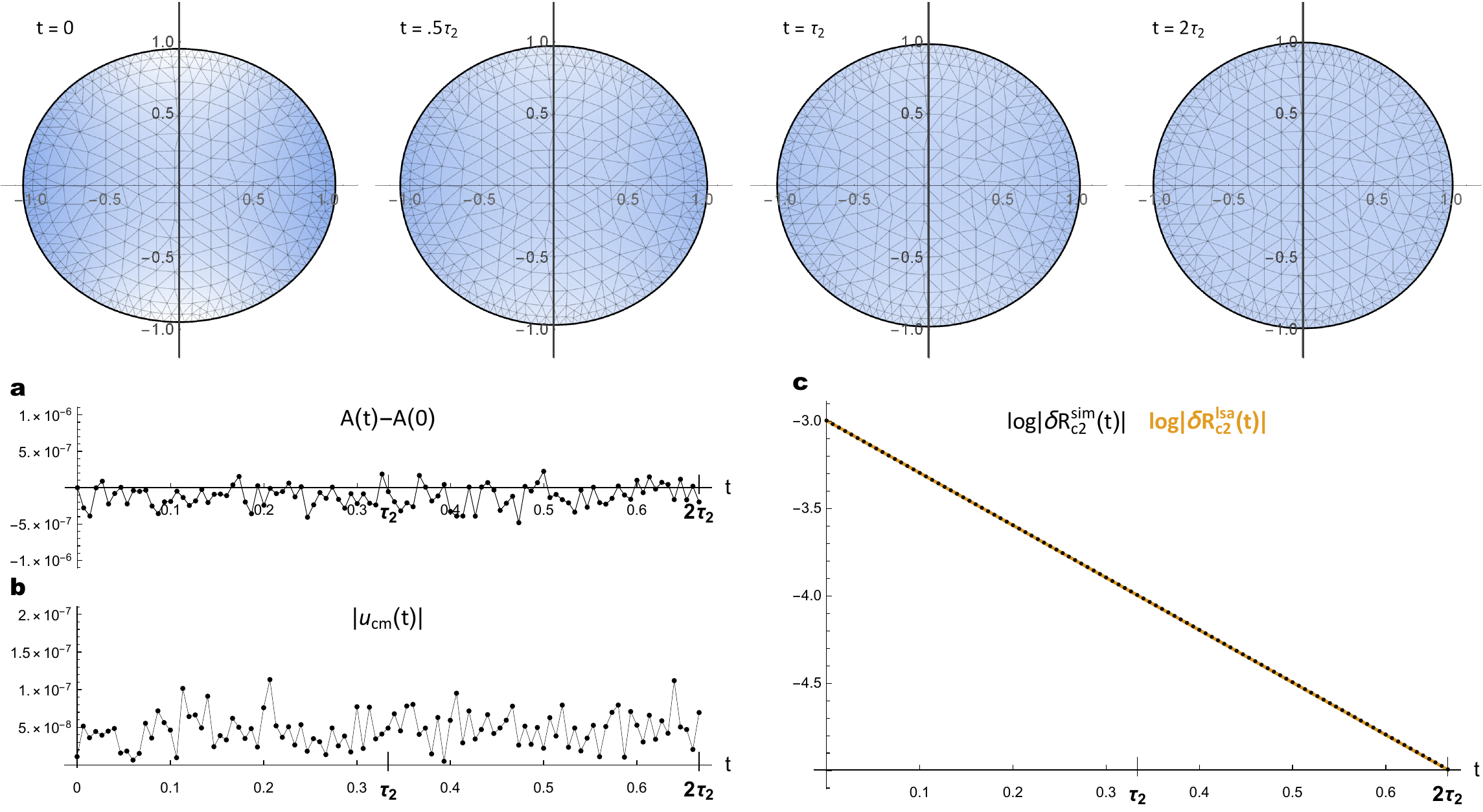}
		\caption{\textbf{Decaying shape perturbation (mode $m=2$)}: We set $\delta R_{\text{c}2}^0=.05$ in Eq. \eqref{eq:HS shape perturbation}, fixed $\sigma=0.5$, and chose $\Delta t=0.0005\tau_2$, where $\tau_2=|s_2|^{-1}=1/3$. 
			At the top we present simulation snapshots demonstrating the decay of the shape perturbation. The density plot in the bulk represents the pressure $p$ (from low in white to high in blue). (a) Time series for the domain area $A(t)$. (b) Time series for the absolute center of mass velocity $|\bu_\text{cm}(t)|$. (c) The predicted linear behaviour of the perturbed mode, $\log|\delta R_{\text{c}2}^\text{lsa}(t)|
			=\log|\delta R_{\text{c}2}^0|-3 t$ (continuous orange line) vs. the simulation time series, $\log|\delta R_{\text{c}2}^\text{sim}(t)|$ (black, computed via Eq. \eqref{eq:sim Fourier components}). We find that the fitted growth rate, $s_2^\text{N}\simeq-2.995$, compares well with the classical linear-stability growth rate, $s_2=-3$ (giving a deviation of $(s_2^\text{N}-s_2)/s_2\sim 2\times 10^{-3}$).\label{fig:ver HS m2}}
	\end{center}
\end{figure}

\begin{figure}
	\begin{center}
		\includegraphics[scale=0.6]{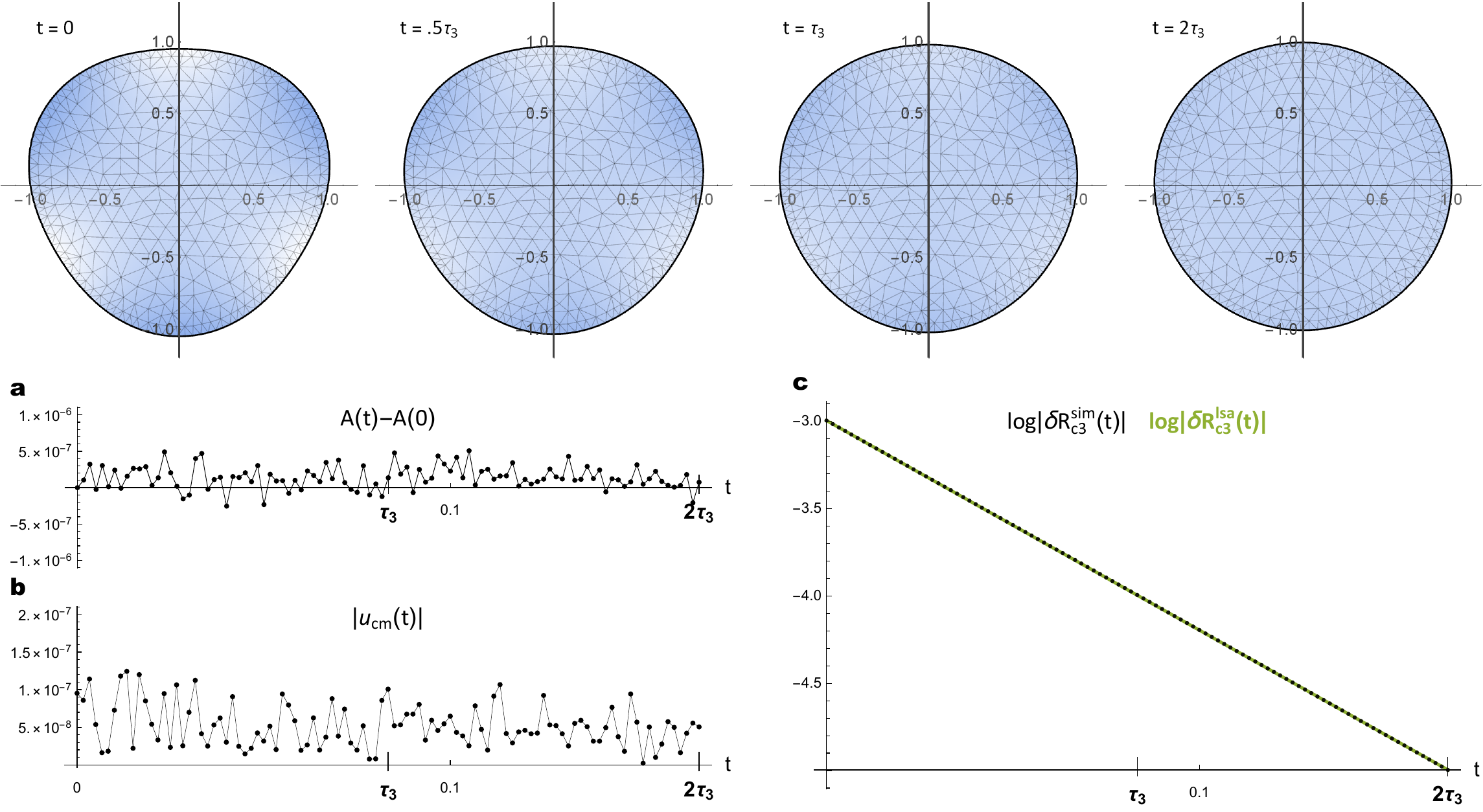}
		\caption{\textbf{Decaying shape perturbation (mode $m=3$)}: We set $\delta R_{\text{c}3}^0=.05$ in Eq. \eqref{eq:HS shape perturbation}, fixed $\sigma=0.5$, and chose $\Delta t=0.0005\tau_3$, where $\tau_3=|s_3|^{-1}=1/12$. 
			At the top we present simulation snapshots demonstrating the decay of the shape perturbation. The density plot in the bulk represents the pressure $p$ (from low in white to high in blue). (a) Time series for the domain area $A(t)$. (b) Time series for the absolute center of mass velocity $|\bu_\text{cm}(t)|$. (c) The predicted linear behaviour of the perturbed mode, $\log|\delta R_{\text{c}3}^\text{lsa}(t)|
			=\log|\delta R_{\text{c}3}^0|- 12 t$ (continuous green line) vs. the simulation time series, $\log|\delta R_{\text{c}3}^\text{sim}(t)|$ (black, computed via Eq. \eqref{eq:sim Fourier components}). We find that the fitted growth rate, $s_3^\text{N}\simeq-11.996$, compares well with the classical linear-stability growth rate, $s_3=-12$ (giving a deviation of $(s_2^\text{N}-s_2)/s_2\sim 4\times 10^{-4}$).\label{fig:ver HS m3}}
	\end{center}
\end{figure}

\begin{figure}
	\begin{center}
		\includegraphics[scale=0.6]{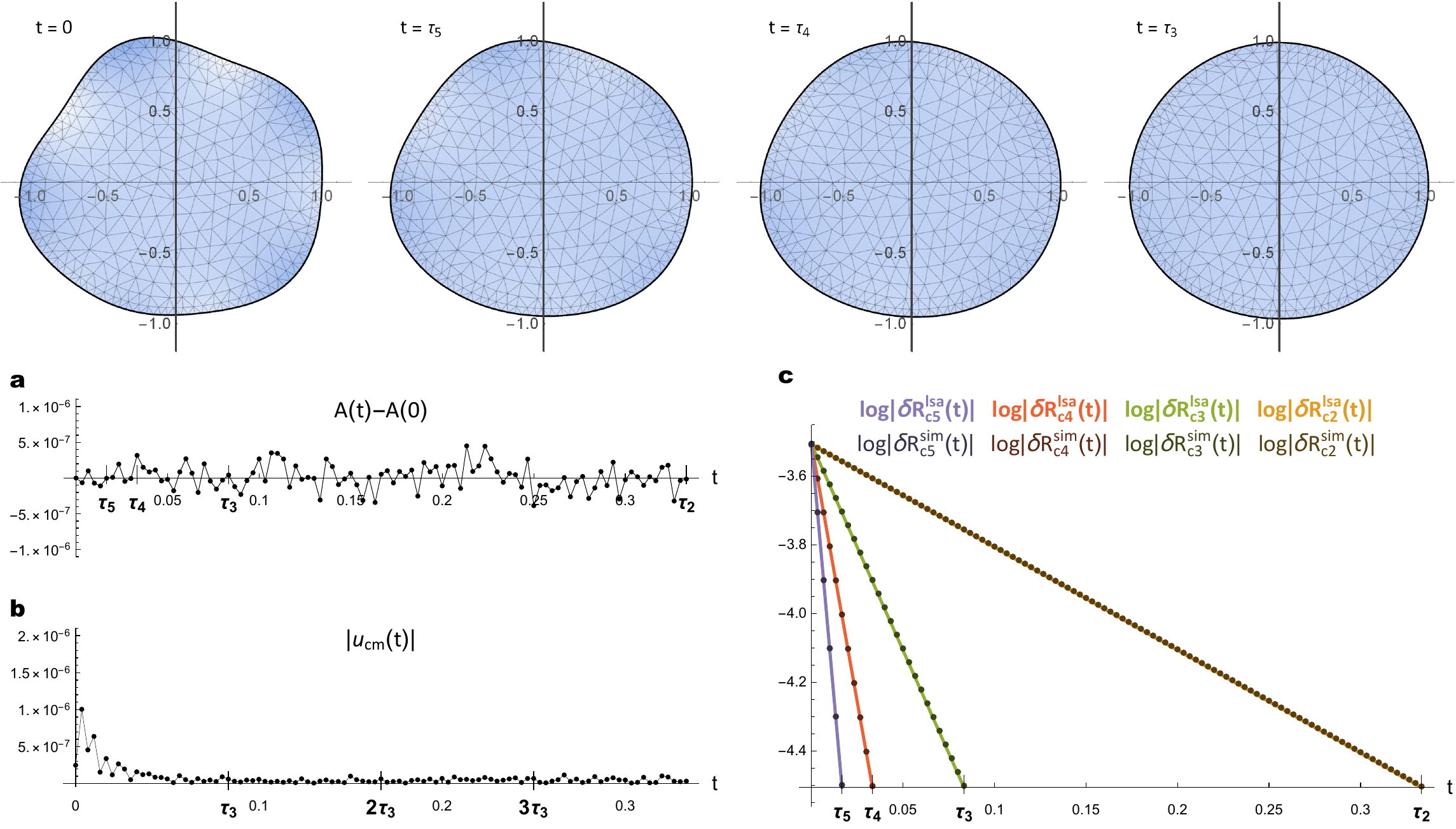}
		\caption{\textbf{Decaying shape perturbations (mixed modes $m=2$--$5$)}: We set $\delta R_{\text{c}2}=.03$, $\delta R_{\text{s}3}^0=-.03$, $\delta R_{\text{s}4}=.03$, and $\delta R_{\text{c}5}=-.03$ in Eq. \eqref{eq:HS shape perturbation}, fixed $\sigma=0.5$, and chose $\Delta t=0.005\tau_5$, where $\tau_5=|s_5|^{-1}=1/120$. 
			At the top we present simulation snapshots demonstrating the sequential decay of the perturbed shape modes. The density plot in the bulk represents the pressure $p$ (from low in white to high in blue). (a) Time series for the droplet area $A(t)$. (b) Time series for the absolute center of mass velocity $|\bu_\text{cm}(t)|$. (c) The predicted linear behaviour of each perturbed mode, $\log|\delta R_{\text{c,s}m}^\text{lsa}(t)|
			=\log|R_{\text{c,s}m}^0|-s_m t$ (continuous colored lines) vs. the simulation time series, $\log|\delta R_{\text{c,s}m}^\text{sim}(t)|$ (darker colors, computed via Eq. \eqref{eq:sim Fourier components}). We find that the fitted growth rates, ($s_2^\text{N}\simeq-2.99$, $s_3^\text{N}\simeq-11.91$, $s_4^\text{N}\simeq -29.83$, $s_5^\text{N}\simeq -59.50$) are all in good quantitative agreement with the classical linear-stability growth rates ($s_2=-3$, $s_3=-12$, $s_4=-30$, $s_5=-60$). 
			\label{fig:ver HS mixed}}
	\end{center}
\end{figure}

Results of three simulations are represented in Figs. \ref{fig:ver HS m2} -- \ref{fig:ver HS mixed}. 
In Figs. \ref{fig:ver HS m2} and \ref{fig:ver HS m3} we introduced an initial shape perturbation strictly in one Fourier mode ($m=2$ and $m=3$, respectively), whereas in Fig. \ref{fig:ver HS mixed} we introduced a superposition of small perturbations in $m=2$--$5$. For the time step and mesh density chosen, we find numerical deviations in $A(t)$ and $|\bu_\text{cm}(t)|$ as low as order $10^{-7}$. Moreover, the fitted numerical growth rates of the perturbed Fourier modes are also in good quantitative agreement with the classical dispersion relation $s_m$ (see details in figure captions). We stress that small deviations in the fitted growth rates may also arise from nonlinear effects which have been neglected in the calculation of $s_m$. As expected, we found through further experimentation with the numerics that precision is gained by decreasing $\Delta t$, increasing the overall mesh density and/or decreasing the initial perturbation amplitudes. 

\begin{figure}
	\begin{center}
		\includegraphics[scale=0.6]{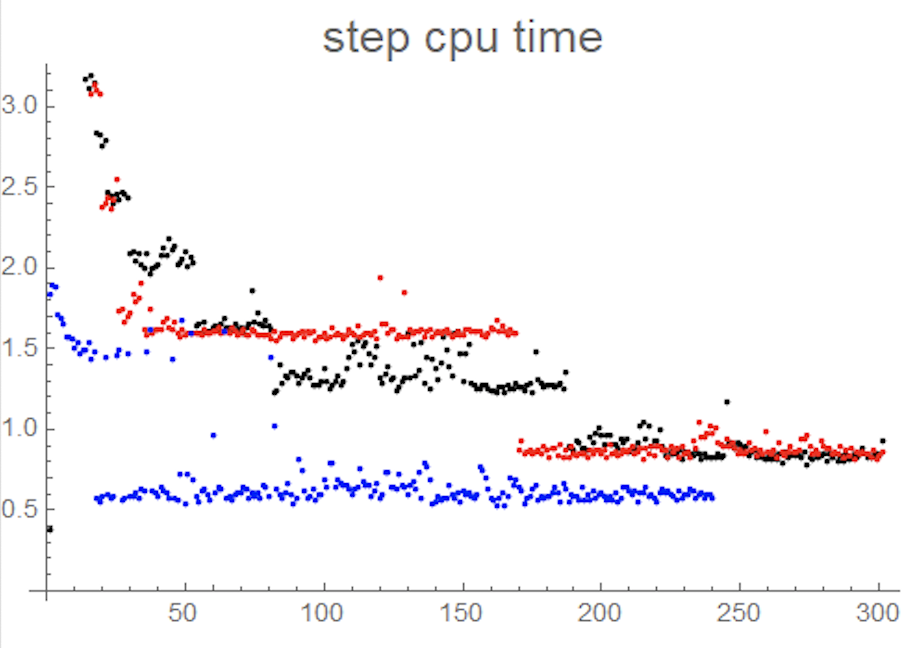}
		\caption{\textbf{Convergence time for the different schemes}
			\label{fig:cpu}}
	\end{center}
\end{figure}

\appendix

\section{De Rham's Theorem}\label{app-a}

In the setting of partial differential equations, the De Rham's theorem solves an over-determined system of linear partial differential equations of order one. 

We recall de Rham's theorem in the case of homogeneous flows (or currents) of dimension one on a Euclidean space:
\begin{theo}[de Rahm]\label{th:deRahm}
	Let $L : L^2(\Omega) ^2\to \R$ be a continuous linear form such that $L(\bv)=0$ for all $\bv \in L^2(\Omega)^2$ with $\nabla \cdot \bv =0$. Then, there exists $q \in H_0^1(\Omega)$ such that for all $\bv \in L^2(\Omega)^2$ 
	\[
	L(\bv)= \int _{\Omega} \nabla q \cdot \bv  \D \bx\, .
	\]
\end{theo}

\section{Linear stability analysis of the Hele-Shaw model}\label{app-b}

The Hele-Shaw model writes 
\begin{equation}
	-\nabla \cdot u=\Delta p =0  \qquad \textrm{ in } \Omega(t) \, ,\label{eq:pression_HS}
\end{equation}
where $\Omega(t)$ is the domain occupied by the fluid and $u=-\nabla p$ is the fluid velocity.

The dynamic boundary condition (or normal force balance) is the given by the Young-Laplace pressure drop
\begin{equation}
	p =\sigma \kappa \qquad  \textrm{ on } \partial \Omega(t) \, ,\label{eq:pression_bord_HS}
\end{equation}
where $p$ is the fluid pressure (minus a constant), $\sigma$ is the surface tension and $\kappa$ the local curvature. The free-boundary evolves with the kinematic condition, which states that the normal velocity of the sharp interface equals the normal velocity of the fluid $ V_n =u\cdot \bn $, recalling that $u=-\nabla p$, in terms of the pressure this gives
\begin{equation}
	V_n =-\nabla p \cdot \bn  \qquad \textrm{ on } \partial \Omega(t) \label{eq:vitesse_HS}\, ,
\end{equation}
where $\bn$ is the unit vector pointing outward.

We start with the stationary solution of a disk $\Omega_0 = \{(r, \theta)|r \le  R_0\}$. The curvature across the boundary is $\kappa = 1/R_0$ and the solution of the pressure is simply $p_0 = \sigma /R_0$ in $\Omega_0$.

We perturb the edge of the domain so that it is defined in terms of the polar angle 
\[R(\theta, t) = R_0  + \varepsilon R_1(\theta, t)\, .\]
We want to analyze how the perturbation $R_1(\theta, t)$ evolves in time; in particular we consider perturbations of the form $R_1(\theta, t) = R_m(t) \cos(m\theta)$, where $R_m(t)$ is the amplitude and $m$ is the wave number. If, for a given $m$, $R_m(t)$ grows in time, that particular wave number is unstable; if $R_m(t)$ decreases, that wave number is stable.

Since $\varepsilon$ is small, we will neglect all the terms which are proportional to $\varepsilon ^n$, $n > 1$. Assume that the pressure can be expanded as follows:
\[p(r,\theta, t) = p_0 + \varepsilon p_1(r,\theta, t)\, ,\]
with $p_1(r,\theta, t) = p_m(r, t) \cos(m\theta)$.

Since $p_0$ satisfies Laplace equation, $p_1$ will satisfy the Laplace equation $\Delta p_1 = 0$. In polar coordinates, this writes
\begin{equation}
	\left( \partial_{r}^2 + r^{-1}\partial_r  - r^{-2}m^2  \right)p_m(r, t) = 0 \, ,
\end{equation}
which is solved by 
\[
p_m(r,t) = A_m(t)r^m +B_m(t)r^{-m}\, .
\] 
Discarding singularities at $r = 0$ we set $B_m(t) = 0$.

Concerning the curvature a direct computation gives that
\[
\kappa= \frac{1}{R_0} +\varepsilon \kappa_1  \, , 
\]
with 
\begin{equation}\label{eq:pert_courbure}
	\kappa_1=-\frac{R_1+\partial_\theta ^2 R_1}{R_0^2}=\frac{m^2-1}{R_0^2}R_m(t)\cos(m\theta)\, .
\end{equation}

From the dynamic boundary condition \eqref{eq:pression_bord_HS} it follows that 
\[ 
(p_0 + \varepsilon p_1)_{|R_0(1+\varepsilon R_1)} = \sigma \kappa_{|R_0(1+\varepsilon R_1)} = \sigma \left(  \frac{1}{R_0} +\varepsilon \frac{m^2-1}{R_0^2}R_m(t)\cos(m\theta) \right)\, ,
\]
hence
\begin{equation}\label{eq:A_m}
	p_{1|R_0(1+\varepsilon R_1)}= \sigma  \frac{m^2-1}{R_0^2}R_m(t)\cos(m\theta) \, .
\end{equation}
On the other hand, neglecting all the terms which are proportional to $\varepsilon ^n$, $n > 1$, we obtain
\[
p_{1|R_0(1+\varepsilon R_1)} = p_m(r=R_0,t) \cos (m\theta)=A_m(t)R_0^m \, ,  
\]
hence
\[
A_m(t)= \frac{\sigma (m^2-1)}{R_0^{m+2}} R_m(t)\cos(m\theta)\, . 
\]

We can now analyze the time evolution of the interface. From the kinematic condition \eqref{eq:vitesse_HS} it follows that $\varepsilon \frac{\D }{\D t} R_1 = -\varepsilon \partial _r p_1$ on the boundary. Hence,
\[
\frac{\D }{\D t} R_1= -\partial_r \left(A (t)r^m\right)_{|r=R_0} = -\sigma R_0^{-3}m(m^2- 1)R_m(t)\cos(m\theta)\, .
\]
Consequently, one has
\[
\frac{\D }{\D t} R_m= -\partial_r \left(A (t)r^m\right)_{|r=R_0} = -\sigma R_0^{-3}m(m^2- 1)R_m(t)\, .
\]

The cubic dispersion relation $\omega = -R^{-3}\sigma m(m^2 - 1)$ shows that the modes $m = 0$ (expansion $m_0$
of the circular droplet) and $m = 1$ (infinitesimal translation of the circular droplet) are marginally stable, alluding to mass conservation and translational symmetry. On the other hand, all $m \ge  2$ modes (morphological deformations) are stabilized by the surface tension $\sigma$.

Indeed, the solution of the last equation is then
\[
R_m(t) = R_m(t = 0) e^{\omega t}\, .
\]
The quantity $\omega$ is often called growth rate, since it determines the "growth" of the perturbation $R_m$. If $\omega > 0$, $R_m$ grows, and the perturbation is unstable; otherwise, it is stable. Stability obviously depends on the wave number of the perturbation, with the perturbations characterized by large $m$ being stable (due to the term proportional to $-m^3$). This is what we expect due to known stabilizing effect of surface tension.

This result expresses the competition between the destabilizing effect of viscosity contrast (destabilizing), and the surface tension (stabilizing).

\bibliographystyle{plain}
\def\cprime{$'$} \def\lfhook#1{\setbox0=\hbox{#1}{\ooalign{\hidewidth
  \lower1.5ex\hbox{'}\hidewidth\crcr\unhbox0}}} \def\cprime{$'$}
  \def\cprime{$'$} \def\cprime{$'$} \def\cprime{$'$} \def\cprime{$'$}

\end{document}